\documentclass[12pt]{amsart}

\usepackage{amsfonts,amsthm,amsmath,amssymb,latexsym}
\usepackage{graphicx}
\graphicspath{ {./figures/} }
\usepackage{eucal}
\usepackage[normalem]{ulem}
\usepackage{xcolor}
\usepackage{bm}
\usepackage{mathtools}

\title{On the derivatives of the Liouville currents}
\address{XD: Department of Mathematics, Kingsborough Community College of the City University of New York, 2001 Oriental Blvd, Brooklyn, NY 11235}
\email{Xinlong.Dong@kbcc.cuny.edu}
 
\author{Xinlong Dong, Dragomir \v Sari\' c and Zhe Wang}
\address{DS: Department of Mathematics, Queens College of the City University of New York, 65--30 Kissena Blvd., Flushing, NY 11367}
\address{DS: Mathematics PhD. Program, Graduate Center of the City University of New York, 365 Fifth Avenue, New York, NY 10016-4309}
\email{Dragomir.Saric@qc.cuny.edu}

\address{ZW: Department of Mathematics, Bronx Community College of the City University of New York, 2155 University Ave, Bronx, NY 10453}
\email{zhe.wang@bcc.cuny.edu}

\thanks{The second author was partially supported by the collaboration grant  346391 from the Simons Foundation and a PSC-CUNY research grant.}

\newtheorem{thm}{Theorem}

\newtheorem{cor}[thm]{Corollary}
\newtheorem{lem}[thm]{Lemma}

\theoremstyle{definition}

\newtheorem{rem}[thm]{Remark}

\theoremstyle{plain}

\newcommand{\D}{\mathbb D}
\newcommand{\C}{\mathbb C}

\newcommand{\R}{\mathbb R}
\newcommand{\Z}{\mathbb Z}
\newcommand{\HH}{\mathbb H}

\newcommand{\T}{\mathcal  T}

\newcommand{\HHH}{\mathcal H}
\newcommand{\W}{\mathbf  W}
\newcommand{\LLL}{\mathcal  L}

\renewcommand{\leq}{\leqslant}
\renewcommand{\geq}{\geqslant}
\renewcommand{\epsilon}{\varepsilon}
\renewcommand{\phi}{\varphi}

\makeatletter
\newcommand{\leqnomode}{\tagsleft@true\let\veqno\@@leqno}
\newcommand{\reqnomode}{\tagsleft@false\let\veqno\@@eqno}


\date{\today}

\begin{document}

\begin{abstract}
The Liouville map, introduced by Bonahon, assigns to each point in the Teichm\"uller space a natural Radon measure on the space of geodesics of the base surface. The Liouville map is real analytic and it even extends to a holomorphic map of a neighborhood of the Teichm\"uller space in the Quasi-Fuchsian space of an arbitrary conformally hyperbolic Riemann surface. The earthquake paths and by their extension quake-bends, introduced by Thurston, are particularly nice real-analytic and holomorphic paths in the Teichm\"uller and the Quasi-Fuchsian space, respectively. We find a geometric expression for the derivative of the Liouville map along earthquake paths. 
\end{abstract}

%%%%%%%%%%%%%%%
\maketitle
%%%%%%%%%%%%%%%

Denote by $X$ an arbitrary conformally hyperbolic Riemann surface. In particular, $X$ can be the upper half-plane $\HH$ or a surface with an infinitely generated fundamental group or an infinite area surface with finitely generated fundamental group, or a finite area surface. The universal covering $\widetilde{X}$  of $X$ is isometric to the upper half-plane $\HH =\{ z=x+iy:y>0\}$ and $X$ is isometrically equivalent to $\HH /\Gamma$ for a Fuchsian group $\Gamma$. The Teichm\"uller space $\T (X)$ is the space of quasiconformal maps from the base surface $X$ onto variable Riemann surfaces up to isometries and homotopies.

The space of geodesics $G(\widetilde{X})$ of the universal covering $\widetilde{X}$ supports a natural {\it Liouville measure} which is the unique (up to scalar multiple) measure of full support that is invariant under the isometries of $\widetilde{X}$ which completely determines the Riemann surface $X$. In general, a $\pi_1(X)$-invariant Radon measure on $G(\widetilde{X})$ is called a {\it geodesic current} for $X$.   
Bonahon \cite{Bonahon} introduced the {\it Liouville map} $$\LLL :\T (X)\to \mathcal{G}(X)$$ from the Teichm\"uller space $\T (X)$ to the space of geodesic currents $\mathcal{G} (X)$ that assigns to each 
quasiconformal deformation $[f:X\to Y] \in \T (X)$ of the Riemann surface X the pullback of the Liouville measure of $Y$ under the deformation $f$.

When the Riemann surface $X$ is compact, Bonahon \cite{Bonahon} used the Liouville map to introduce an alternative description of the Thurston boundary to $\T (X)$. The second author \cite{Saric}, and more recently, Bonahon and the second author \cite{BonahonSaric} introduced the Thurston boundary to Teichm\"uller spaces of arbitrary conformally hyperbolic Riemann surfaces.  
Using the description of the topology on the space of bounded geodesic currents $\mathcal{G}_b (X)$ for a Riemann surface  $X$ given in \cite{BonahonSaric}, the first two authors  introduced a space of bounded H\"older distributions $\mathcal{H}_b(X)$ (see \cite{DongSaric} and \S \ref{sec:Holder}) which contains bounded geodesic currents and they proved that the Liouville map
$$
\mathcal{L} :\T (X)\to\mathcal{H}_b(X)
$$
is real analytic. Prior to \cite{DongSaric}, Bonahon and S\" ozen \cite{BonahonSozen} proved that the Liouville map is differentiable when $X$ is a compact surface, the second author \cite{Saric} extended this result to all conformally hyperbolic Riemann surfaces. Using Bers' simultaneous uniformization, Otal \cite{Otal} proved that the Liouville map is real-analytic in the topology introduced in \cite{Saric1} by extending the Liouville map to an open neighborhood of $\mathcal{T}(X)$ inside the QuasiFuchsian space $\mathcal{QF}(X)$, where $\mathcal{T}(X)$ is realized as a totally real analytic submanifold of $\mathcal{QF}(X)$.  The space of bounded H\"older distributions $\mathcal{H}_b(X)$ simplifies the description of the topology in \cite{Saric1}.

In addition to the Liouville currents, a conformally hyperbolic Riemann surface supports measured laminations. Thurston \cite{Thurston} introduced a natural deformation of Riemann surfaces by left shearing along geodesics of the support of a measured lamination $\mu$ with the amount given by the transverse measure called an {\it earthquake map} $E^{\mu}:X\to X^{\mu}$ (see \S \ref{sec:earthquakes}). Thurston \cite{Thurston} proved that any homeomorphic deformation of $X$ onto another Riemann surface can be obtained by a unique earthquake. A quasiconformal deformation of a Riemann surface is obtained by a unique earthquake $E^{\mu}$ with {\it bounded} measured lamination (see \cite{Saric2}). In addition, Miyachi and the second author \cite{MiyachiSaric} proved that the natural correspondence between the quasiconformal deformations of $X$ (namely, the Teichm\"uller space $\T (X)$) and the space of bounded measured lamination is a homeomorphism. 

Let $\widetilde\mu$ be the lift of $\mu$. If $\widetilde\mu$ is a bounded measured lamination and $t>0$, then $t\widetilde\mu$ is also a bounded measured lamination. Therefore $t\mapsto E^{t\widetilde\mu}$ is a path in the Teichm\"uller space called an {\it earthquake path}. The second author \cite{Saric2} proved that an earthquake path is a real analytic path in the Teichm\"uller space $\T (X)$. Therefore the composite path 
$$
t\mapsto \mathcal{L}(E^{t\widetilde\mu})
$$
is a real analytic path in $\mathcal{H}_b(X)$ and we compute its tangent vector.

\begin{thm}
\label{thm:earthq-der-intro}
Let $\mu$ be a bounded measured lamination on a conformally hyperbolic Riemann surface $X$ and let $\mathcal{L}:\T (X)\to\mathcal{H}_b(X)$ be the Liouville map. The image of the tangent vector $\dot{E}^{\widetilde\mu}:=\frac{d}{dt}E^{t\widetilde\mu}|_{t=0}$ to the earthquake path $t\mapsto E^{t\widetilde\mu}$ is given by the formula
$$
d\mathcal{L}(\dot{E}^{\widetilde\mu})(\xi)=\int_{G(\widetilde{X})} \int_{G(\widetilde{X})} \xi (h)\cos (g,h)dL_{\widetilde{X}}(h)d\widetilde{\mu}(g)
$$
where $\widetilde{\mu}$ is the lift of $\mu$ to $\widetilde{X}$, $h \in G(\widetilde{X})$, $g \in G(\widetilde{X})$, $\xi :G(\widetilde{X})\to\mathbb{C}$ is a H\"older continuous function with compact support and $L_{\widetilde{X}}$ is the Liouville measure on $G(\widetilde{X})$.
\end{thm}

Each tangent vector at a point of the Teichm\"uller space $\T (X)$ is obtained by taking the derivative along an earthquake path passing through that point. Moreover, the tangent space at a point of $\T (X)$ is homeomorphic to the space of bounded measured lamination on the surface corresponding to that point (see \cite{MiyachiSaric}). Theorem \ref{thm:earthq-der-intro} gives an explicit formula in the geometric terms for the tangent map $d\mathcal{L}$ to the Liouville map 
$\mathcal{L}:\T (X)\to\mathcal{H}_b(X)$. 

We call the extension of $\mathcal L$ defined by Otal \cite{Otal} the extended Liouville map $\hat{\mathcal{L}}$. The image under $\hat{\mathcal{L}}$ of the points not in $\mathcal{T}(X)$ consists of complex-valued distributions or, rather, finitely additive complex-valued measures. These objects usually require more restrictive analytic settings to be able to integrate against functions and to be able to take their derivatives. 

For $0<\lambda\leq 1$, let $H^{\lambda}(\widetilde{X})$ be the space of all H\"older continuous functions $\xi :G(\tilde{X})\to\mathbb{C}$ with compact support and H\"older exponent $\lambda$. Thurston introduced a natural complexification of earthquakes called quake-bends. 
The naturality of the geometric setting and the holomorphic motions allows us to extend the formula (given by a limit) for the first derivative of the Liouville map to the neighborhood of $\mathcal{T}(X)$ in $\mathcal{QF}(X)$ along the quake-bends. 

\begin{thm}
\label{thm:intro-quake-der}
Let $\xi\in H^{\lambda}(\widetilde{X})$ and let $\delta >0$ be the radius for which $\hat{\mathcal{L}}([E^{\tau \widetilde{\mu}}])(\xi )$ is defined. Then, for $\tau\in\mathbb{C}$ with $|\tau |<\delta$,
$$
\frac{d}{d\tau} \hat{\mathcal{L}}([E^{\tau \tilde{\mu}}])(\xi )=\lim_{j\to\infty}\int_{G(\tilde{X})}\int_{G(\tilde{X})}\xi (h)\cosh \bm{d}(E^{\tau\widetilde{\mu}|_{K_j}}(g),E^{\tau\widetilde{\mu}|_{K_j}}(h))dL_{[E^{\tau\widetilde{\mu}|_{K_j}}]}(h)d\tilde{\mu}|_{K_j}(g)
$$ where $\widetilde{\mu}|_{K_j}\to\widetilde{\mu}$ in the weak* topology as $j\to\infty$ (see \S 4) and $\bm{d}$ is the complex distance.
\end{thm}

There is also a nice formula for the second derivative of $\mathcal{L}(E^{t\tilde{\mu}})$ at $t=0$. However, there is no corresponding second derivative formula (given by a limit) along quake-bends, see $\S 6$.
\begin{thm}
	\label{thm:intro-quake-second-der}
	Let $\mu$ be a bounded measured lamination on a conformally hyperbolic Riemann surface $X$ and let $\mathcal{L}:\T (X)\to\mathcal{H}_b(X)$ be the Liouville map. Then,
	\begin{flalign*}
		\frac{d^2}{dt^2}{\mathcal{L}}([E^{t\widetilde{\mu}}])(\xi ) \Big{|}_{t=0}
		&=\int_{G(\widetilde{X})} \int_{G(\widetilde{X})} \\
		&\quad \bigg\{\int_{G(\widetilde{X})}\xi (h) \big[\cos(g, h)\cos(g', h) - \frac{1}{2} \sin(g, h)\sin(g', h) e^{-d_h}\big]dL_{\widetilde X}(h)\bigg\}\\
		&\quad d\widetilde{\mu}(g) d\widetilde{\mu}(g')
	\end{flalign*} where $\widetilde{\mu}$ is the lift of $\mu$ to $\widetilde{X}$ and $d_h$ is the hyperbolic distance along $h$ from $g \cap h$ to $g' \cap h$.
\end{thm}

\section{The Teichm\"uller space}

Fix a conformally hyperbolic Riemann surface $X$ of possibly infinite hyperbolic area. The Riemann surface $X$ is identified with $\HH /\Gamma$, where $\Gamma <PSL_2(\R )$ is a Fuchsian group acting on the upper half-plane $\HH$. A quasiconformal map $f:\HH\to\HH$ is {\it normalized} if it fixes $0$, $1$ and $\infty$.
Two normalized quasiconformal map $f:\HH \to\HH$ that conjugate the Fuchsian group $\Gamma$ onto another Fuchsian group are {\it Teichm\"uller equivalent} if they agree on the ideal boundary $\hat{\R}=\R\cup\{\infty\}$ of the upper half-plane $\HH$. The Teichm\"uller space $\T (X)$ consists of all Teichm\"uller equivalence classes $[f]$ of normalized quasiconformal maps conjugating $\Gamma$ onto another Fuchsian group. 

The Beltrami coefficient of a quasiconformal map $f:\HH\to\HH$ is given by $\mu =\frac{f_{\bar{z}}}{ f_z}$ and it satisfies $\|\mu\|_{\infty}<1$. Conversely, given $\mu\in L^{\infty}(\HH )$ with $\|\mu\|_{\infty}<1$ there exists a unique (normalized) quasiconformal map $f:\HH \to\HH$ that fixes $0$, $1$ and $\infty$ whose Beltrami coefficient is $\mu$. A quasiconformal map $f$ conjugates $\Gamma$ onto another Fuchsian group if and only if $\mu\circ\gamma\frac{\overline{\gamma'}}{\gamma'}=\mu$ for all $\gamma\in \Gamma$. 
Two Beltrami coefficients are {\it Teichm\"uller equivalent} if the corresponding normalized quasiconformal maps are equal on $\hat{\R}$.  Therefore we can define $\T (X)$ to be a set of all Teichm\"uller classes $[\mu]$ of Beltrami coefficients that satisfy $\mu\circ\gamma (z)\frac{\overline{\gamma'(z)}}{\gamma'(z)}=\mu (z)$ for all $\gamma\in \Gamma$ and $z\in\HH$ (for example, see \cite{GardinerLakic}).

A normalized quasiconformal map $f:\hat{\mathbb{C}}\to\hat{\mathbb{C}}$ that conjugates  $\Gamma<PSL_2(\mathbb{R})$ onto a subgroup of $PSL_2(\mathbb{C})$ represents an element of the quasi-Fuchsian space $\mathcal{QF}(\Gamma )$. Two normalized quasiconformal maps $f$ and $g$ that conjugate $\Gamma$ onto a subgroup of $PSL_2(\C )$ are {\it equivalent} if they agree on $\hat{\R}=\R\cup\{\infty\}$. Denote by $[f]\in\mathcal{QF}(\Gamma )$ the corresponding equivalence class. Equivalently, we can define $\mathcal{QF}(\Gamma )$ to consist of all equivalence classes $[\mu ]$ of Beltrami coefficients on $\C$ where two Beltrami coefficients are equivalent if their corresponding normalized quasiconformal maps agree on $\hat{\R }$. 
If $\Gamma$ is trivial then $f$ is a quasiconformal map fixing $0$, $1$ and $\infty$ which does not necessarily preserve the upper half-plane $\HH$. When $X=\HH /\Gamma$, then we set $\mathcal{QF}(X)=\mathcal{QF}(\Gamma )$.

Bers introduced a complex Banach manifold structure to the Teichm\"uller space $\T (X)$. The complex chart around the basepoint $[0 ]\in \T (X)$ is obtained as follows. Let $\widetilde{\mu}$ be the Beltrami coefficient which equals $\mu$ in the upper half-plane $\HH$ and equals zero in the lower half-plane $\HH^{-}$. The solution $f=f^{\widetilde{\mu}}$ to the Beltrami equation $ f_{\bar{z}}=\widetilde{\mu}  f_z$  is conformal in the lower half-plane $\HH^{-}$. The Schwarzian derivative 
$$
S(f^{\widetilde{\mu}})(z)=\frac{(f^{\widetilde{\mu}})'''(z)}{(f^{\widetilde{\mu}})'(z)}-\frac{3}{2}\Big{(}\frac{(f^{\widetilde{\mu}})''(z)}{(f^{\widetilde{\mu}})'(z)}\Big{)}^2
$$
for $z\in\HH^{-}$ defines a holomorphic function $\phi (z)=S(f^{\widetilde{\mu}})(z)$ which satisfies
$(\phi \circ\gamma )(z) \gamma'(z)^2=\phi (z)$ and $\|\varphi\|_{b}:=\sup_{z\in\HH^{-}}|y^2\phi (z)|<\infty$, called a {\it cusped form} for $X$. The space of all cusped forms $\phi :\HH^{-}\to\C$ for $X$ is a complex Banach space $\mathcal{Q}_b(X)$ with the norm $\|\cdot\|_b$ (see \cite{GardinerLakic}).  

The Schwarzian derivative maps the unit ball in $L^{\infty}(\HH )$ onto an open subset of $\mathcal{Q}_b(X)$ and it projects to a homeomorphism $\Phi$ from $\T (X)$ to an open subset of $\mathcal{Q}_b(X)$ containing the origin. 
The open ball $B_{[0]}(\frac{1}{2}\log 2)$ in $\T (X)$ of radius $\frac{1}{2}\log 2$ and center $[0]$ maps under $\Phi$ onto an open set in $\mathcal{Q}_b(X)$ which contains the ball of radius $\frac{2}{3}$ and is contained in the ball of radius $2$ with center $0\in\mathcal{Q}_b(X)$. The map $\Phi :B_{[0]}(\frac{1}{2}\log 2)\to \mathcal{Q}_b(X)$ is a chart map for the base point $[0]\in\T (X)$ (see \cite[\S 6]{GardinerLakic}). The Ahlfors-Weill section provides an explicit formula for $\Phi^{-1}$ on the ball of radius $\frac{1}{2}$ and center $0$ in $\mathcal{Q}_b(X)$. Namely if $\phi\in \mathcal{Q}_b(X)$ with $\|\phi\|_{b}<\frac{1}{2}$ then Ahlfors and Weill prove that $\Phi^{-1}(\phi )=[-2y^2\phi (\bar{z})]$ (see \cite[\S 6]{GardinerLakic}). The Beltrami coefficient $\eta_{\phi}(z):=-2y^2\phi (\bar{z})$  is said to be {\it harmonic}. The Ahlfors-Weill formula gives an explicit expression of Beltrami coefficients that are representing points in $\T (X)$ corresponding to the holomorphic disks $\{t\phi :|t|<1,\|\phi\|_{b}<\frac{1}{2}\}$ in the chart in $\mathcal{Q}_b(X)$, namely $$\Phi^{-1}(\{t\phi :|t|<1\})=\{ [t\eta_\phi ]\in\T (X) :|t|<1\}.$$

By the Bers simultaneous uniformization theorem, the Quasi-Fuchsian space $\mathcal{QF}(X)$ is identified with $\T (X)\times \T (\bar{X})$, where $\bar{X}$ is the Riemann surface which is anti-conformal to $X$. A complex chart for $\mathcal{QF}(X)$ at the basepoint $[0]$ is the product of two open balls in $\mathcal{Q}_b(X)$ and in $\mathcal{Q}_b(\bar{X})$ with centers at the origins. The Teichm\"uller space $\T (X)$ embeds as a totally real submanifold of $\mathcal{QF}(X)$.

\section{The Liouville map and uniform H\"older topology}
\label{sec:Holder}

In this section we define the Liouville map, the space of bounded geodesic currents and the space of bounded H\"older distributions for the Riemann surface $X$ (see \cite{Bonahon}, \cite{BonahonSaric} and \cite{DongSaric}).

Recall that the conformally hyperbolic Riemann surface $X$ is identified with $\HH /\Gamma$, where $\Gamma$ is a Fuchsian group acting on the upper half-plane $\HH$. The space of oriented geodesics $G(\widetilde{X})=(\partial_{\infty}\widetilde{X}\times \partial_{\infty}\widetilde{X})\setminus\mathrm{diagonal}$ is identified with $(\hat{\R}\times\hat{\R})\setminus\mathrm{diagonal}$. 

The {\it angle distance} 
$d(x,y)$ for $x,y\in \hat{\R}=\R\cup\{\infty\}$ with respect to a reference point $z_0\in\HH$ is the smaller of the two angles between the geodesic rays starting at $z_0$ and ending at $x$ and $y$, respectively. The angle distance $d$ depends on the choice of the reference point $z_0$. The identity map of $\hat{\mathbb{R}}$ is bi-Lipschitz for any two angle distances given by different choices of the reference points.
The angle distance 
induces the product metric on $G(\widetilde{X})$ via the identification  with $(\hat{\R}\times\hat{\R})\setminus\mathrm{diagonal}$ called the {\it angle metric}. The identity map on $G(\widetilde{X})$ is bi-Lipschitz for any two angle metrics given by two different reference points. It is also possible to isometrically identify $\tilde{X}$ with the unit disk model $\mathbb{D}$ of the hyperbolic plane. In that case $G(\tilde{X})$ is identified with $(S^1\times S^1)\setminus\mathrm{diagonal}$.

A {\it geodesic current} for $X$ is a positive Radon measure on $G(\widetilde{X})=(\hat{\R}\times\hat{\R})\setminus\mathrm{diagonal}$ that is invariant under the action of the covering group $\Gamma$ and the change of the orientation of the geodesics (see \cite{Bonahon}). The {\it space of geodesic currents} for $X$ is denoted by $\mathcal{G}(X)$. The {\it Liouville measure} $L_{\widetilde{X}}$ for ${X}$ is the unique (up to scalar multiple) full support geodesic current that is invariant under the isometries of the universal covering $\widetilde{X}$ and the change of the orientation of the geodesics. More precisely, the Liouville measure of a Borel set $A\subset (\hat{\R}\times\hat{\R})\setminus\mathrm{diagonal}=G(\widetilde{X})$ is given by
 $$
 L_{\widetilde{X}}(A)=\iint_{A}\frac{dxdy}{(x-y)^2}.
 $$
 Given two disjoint intervals $[a,b]$ and $[c,d]$ of $\hat{\mathbb{R}}$, 
 the set of geodesics $A=[a,b]\times [c,d]\subset G(\widetilde{X})$  with one endpoint in $[a,b]\subset\hat{\mathbb{R}}$ and another endpoint in $[c,d]\subset\hat{\mathbb{R}}$ is called a {\it box of geodesics}. Then
 $$
 L_{\widetilde{X}}([a,b]\times [c,d])=\log cr(a,b,c,d).
 $$ 
 It follows from the definition that $L_{\widetilde{X}}([a,b]\times [c,d])=L_{\widetilde{X}}([c,d]\times [a,b])$.
 
A quasiconformal map $f:X\to X_1$ induces a quasisymmetric map of the ideal boundaries of the universal covers $\widetilde{X}$ and $\widetilde{X}_1$ of $X$ and $X_1$, respectively. The induced map in turn induces a homeomorphism of the space of geodesics $G(\widetilde{X})$ and $G(\widetilde{X}_1)$ which is equivariant for the actions of the covering groups. The Teichm\"uller equivalence class  $[f:X\to X_1]$ induces the pull-back of the Liouville measure $L_{\widetilde{X}_1}$ to the space of geodesics $G(\widetilde{X})$ denoted by $L_{[f]}$ under the induced homeomorphism of $G(\widetilde{X})$ and $G(\widetilde{X}_1)$. In general, $L_{[f]}$ is invariant under $\pi_1(X)$ but not invariant under all isometries of $\widetilde{X}$ and thus different from $L_{\widetilde{X}}$.
 In fact, $L_{[f]}$ completely recovers the Riemann surface $X_1$ and the Teichm\"uller class of $f:X\to X_1$.

 Bonahon \cite{Bonahon} introduced the {\it Liouville map} 
 $$
 \mathcal{L}:\T (X)\to \mathcal{G}(X)
 $$
 by
 $$
 \mathcal{L}([f])=L_{[f]}.
 $$
and he used the Liouville map in order to give an alternative description of the Thurston boundary to the Teichm\"uller space of a compact surface. When $X$ is a compact surface the space of geodesic currents $\mathcal{G}(X)$ is equipped with the standard weak* topology for which the Liouville map is an embedding onto its image (see Bonahon \cite{Bonahon}).
 Bonahon and the second author \cite{BonahonSaric} introduced the {\it uniform weak* topology} on the space of geodesic currents in order to introduce a Thurston boundary to Teichm\"uller spaces of arbitrary Riemann surfaces. This is a simplification of the topology that was introduced by the second author \cite{Saric}.

Let $H(\widetilde{X})$ be the space of all H\"older continuous functions $\xi: G(\widetilde{X})\to\C$ with respect to the product metric on $G(\widetilde{X})$ that are of compact support. 
 A linear functional $\W : H(\widetilde{X})\to\C$ is said to be {\it bounded} if, for every $\xi\in H(\widetilde{X})$,
$$
\| \W \|_{\xi}:=\sup_{\gamma\in PSL_2(\R )} |\W (\xi\circ\gamma )|<\infty .
$$ 
The space of {\it bounded H\"older distributions} $\HHH_{b}(X)$ is the space of all bounded complex linear functionals on the space of H\"older continuous functions $H(\widetilde{X})$ with compact support in $G(\widetilde{X})$.

Since the space of bounded geodesic currents is a subset of the space of bounded H\"older distributions $\HHH_{b}(X)$, we can consider the Liouville map
$$
\mathcal{L} :\T (X)\to \HHH_{b}(X).
$$
The Liouville map is a homeomorphisms onto its image in $\HHH_{b}(X)$ because continuous functions are well approximated by H\"older continuous functions (see \cite[Theorem 4.2.1]{Dong}).

The first two authors \cite{DongSaric} complexified the Liouville map and proved that the complexification is holomorphic. For a fixed $0<\lambda \leq 1$, let $\HHH_{b}^{\lambda}(X)$ be the space of complex linear functional on the space $H^{\lambda}(\widetilde{X})$ of $\lambda$-H\"older continuous functions with compact support that are bounded (for the semi-norms given by the $\lambda$-H\"older continuous functions with compact support). Given $\delta >0$, define $\mathcal{V}_{\delta}$ to be the set of all $[\mu ]\in\mathcal{QF}(X)$ with $\|\mu\|_{\infty}<\delta$. Then (see \cite[Theorem 7 and 8]{DongSaric})

\begin{thm}
\label{thm:complexification}
For a fixed $0<\lambda\leq 1$, there exists $\delta =\delta (\lambda )>0$ such that the Liouville map $\mathcal{L}:\T (X)\to \HHH_b(X)$ extends to a holomorphic map
$$
\hat{\mathcal{L}}:\mathcal{V}_{\delta}\to \HHH_b^{\lambda}(X).
$$
\end{thm}

\section{Earthquakes, quake-bends and tangent vectors}
\label{sec:earthquakes}

Earthquakes are geometrically natural deformations of conformally hyperbolic Riemann surfaces. Thurston first introduced earthquakes on compact hyperbolic surfaces as completions of sequences of real positive twists along longer and longer geodesics (see Kerckhoff \cite{Kerckhoff} and Thurston \cite{Thurston}). 
Later on, Thurston \cite{Thurston-pl} defined earthquakes on the hyperbolic plane and using lifts to the universal covering he extended the definition of earthquakes to any conformally hyperbolic Riemann surface.

A {\it geodesic lamination} on $X$ is a closed subset of $X$ with an assigned foliation by complete geodesics. When $X$ is of finite hyperbolic area any geodesic lamination of $X$ has zero area and its foliation is unique (see \cite{Thurston}). This is also true for infinite area hyperbolic surfaces whose covering group $\Gamma$ is of the first kind (see \cite[Proposition 3.1]{Saric-tt}). However, in general we need to include the foliation in the definition of a geodesic lamination. For example, the hyperbolic plane $\HH$ can be foliated by complete simple geodesics in infinitely many ways.

A {\it measured lamination} $\mu$ on $X$ is a geodesic lamination $|\mu |$ together with an assignment of a positive Radon measure to each geodesic arc $I$ transverse to $|\mu |$ such that the measure is supported on $I\cap |\mu |$ and the assignment is invariant under homotopies relative the leaves of $|\mu |$. Equivalently, we can define a measured lamination $\mu$ on $X$ to be a geodesic current $\widetilde{\mu}\in\mathcal{G}(X)$ for $X$ whose support $|\widetilde{\mu}|$ is a geodesic lamination on $\HH$ (which is necessarily invariant under $\pi_1(X)$). To be precise, we take geodesics of the support  $|\mu |$ with both orientations and make the measure invariant under the change of orientation. 

An earthquake map of $X$ is defined using a measured lamination $\mu$ on $X$. We define the earthquake map $E^{\widetilde{\mu}}:\HH\to\HH$ using the lift $\widetilde{\mu}$ of the measured lamination $\mu$ to the universal covering $\widetilde{X}=\HH$. It is necessarily true that the measured lamination $\widetilde{\mu}$ is $\Gamma$ invariant and is supported on complete geodesics of $\HH$.

We first define a {\it simple} earthquake $E^\delta_g$ supported on a single geodesic $g\in G(\widetilde{X})=G(\HH )$ corresponding to a measured lamination $\delta\mathbf{1}_g+\delta\mathbf{1}_{\widetilde{g}}$, where $\delta >0$, $g$ and $\widetilde{g}$ have opposite orientations, and $\mathbf{1}_g$ is a Dirac measure on $G(\widetilde{X})$ with support $g$-i.e., $\mathbf{1}_g(h)=0$ for $h\neq g$ and $\mathbf{1}_g(g)=1$. Assign an orientation to $g$. The oriented geodesic $g$ divides $\HH$ into the left and right geodesic half-planes.  Define $E^\delta_g:\HH \to\HH$ to be
the identity on the left geodesic half-plane and to be the hyperbolic translation with the oriented axis $g$ and the translation length $\delta$ on the right geodesic half-plane.

To define $E^\delta_g:\T (X)\to\T (X)$, we set $E^\delta_g([id]):=E^\delta_g$. For a quasisymmetric boundary map $\widetilde{f}:\hat{\mathbb{R}}\to\hat{\mathbb{R}}$ corresponding to $[f]\in\T (X)$, we set $E^\delta_g([\widetilde{f}]):=E_{\widetilde{f}(g)}^\delta$ where $\widetilde{f}(g)$ is the geodesic of $\HH$ whose endpoints are the image of endpoints of $g$ under $\widetilde{f}$. This defines a simple earthquake on the whole Teichm\"uller space $\T (X)$ (for example, see \cite{BonahonSaric}).

Let $\{ g_1,\widetilde{g}_1,\ldots ,g_n,\widetilde{g}_n\}$ be a finite set of pairwise disjoint geodesics in $\HH$ where $g_i$ and $\widetilde{g}_i$ have opposite orientations. Let $\{ \delta_1,\ldots ,\delta_n\}$ be positive real numbers. Consider a measured lamination $\sigma =\sum_{i=1}^n\delta_i(\mathbf{1 }_{g_i}+\mathbf{1}_{\widetilde{g}_i})$ and define an {\it elementary earthquake} with measured lamination $\sigma$  
$$
E^{\sigma}:\T (X)\to\T (X)
$$
by (see \cite{BonahonSaric})
$$
[\widetilde{f}]\mapsto E^{\delta_1}_{g_1}\circ E^{\delta_2}_{g_2}\circ\cdots \circ E^{\delta_n}_{g_n}([\widetilde{f}]).
$$
We note that the order of the terms $E^{\delta_i}_{g_i}$ in the above definition is irrelevant due to the fact that we defined $E^{\delta}_g([\widetilde{f}]):=E_{\widetilde{f}(g)}^\delta$ (see \cite{BonahonSaric}). 
Since the support of $\sigma$ is finite, the elementary earthquake extends to a homeomorphism of $\hat{\mathbb{R}}$.

Let $\widetilde{\mu}$ be a measured lamination on $\HH$ that is the lift of a measured lamination $\mu$ on $X$. A {\it stratum} 
of the geodesic lamination $|\widetilde{\mu}|$ is either a geodesic of $|\widetilde{\mu}|$ or a connected component of its complement. 
The upper half-plane $\HH$ is partitioned by strata of $|\widetilde{\mu}|$ and we define the earthquake map $E^{\widetilde{\mu}}:\HH\to\HH$ by assigning a hyperbolic isometry on each stratum. The measured lamination $\widetilde{\mu}$ can be obtained as a limit in the weak* topology of a sequence of finitely supported measured laminations $\delta_n$. For example, one can choose finitely many geodesics of the support of $\widetilde{\mu}$ and assign to each of them a positive weight of all nearby geodesics (see \cite{Thurston-pl}, or \cite{MiyachiSaric}). 
Then, on each stratum of $|\widetilde{\mu} |$, the limit of elementary earthquakes $E^{\delta_n}$ as $n\to\infty$ exists and the earthquake $E^{\widetilde{\mu}}$ restricted to the stratum is defined to be the limit (see \cite{Thurston-pl} and \cite{EpsteinMarden}). 

It turns out that $E^{\widetilde{\mu}}$ does not always extend to a homeomorphism of the ideal boundary $\hat{\mathbb{R}}$ (see \cite{Thurston-pl}, \cite{GardinerHuLakic}). Thurston \cite{Thurston-pl} showed that $E^{\widetilde{\mu}}$ is, up to post-composition by an isometry, uniquely determined by $\widetilde{\mu}$. More importantly,  Thurston \cite{Thurston-pl} proved that every orientation preserving homeomorphism of the ideal boundary $\hat{\mathbb{R}}$ can be obtained by the continuous extension of an earthquake map $E^{\widetilde{\mu}}:\HH\to\HH$. It is an open problem to classify measured laminations whose earthquakes give rise to homeomorphisms of $\hat{\mathbb{R}}$. 

For the considerations involving $\T (X)$, it is important to classify which measured laminations give rise to earthquakes whose continuous extensions to $\hat{\mathbb{R}}$ are quasisymmetric maps. The {\it Thurston norm} of a measured lamination $\widetilde{\mu}$ on $\HH$ is given by
$$
\|\widetilde{\mu}\|_{Th}:=\sup_I\widetilde{\mu}(I)
$$
where the supremum is over all geodesic arcs $I$ of length $1$ that are transverse to $\widetilde{\mu}$. The quantity $\widetilde{\mu}(I)$ is the $\widetilde{\mu}$-measure of the geodesics in $\HH$ that intersect $I$. The second author \cite{Saric}, Gardiner-Hu-Lakic \cite{GardinerHuLakic} and Epstein-Marden-Markovic \cite{EMM}  gave different proofs of the fact that $E^{\widetilde{\mu}}$ extends to a quasisymmetric map of $\hat{\mathbb{R}}$ if and only if $\|\widetilde{\mu}\|_{Th}<\infty$. We note that $\|\widetilde{\mu}\|_{Th}<\infty$ is equivalent to $\widetilde{\mu}\in\mathcal{H}_b(\HH )$. 

When $t>0$ and $\|\widetilde{\mu}\|_{Th}<\infty$, the measured lamination $t\widetilde{\mu}$ has finite Thurston norm and $t\mapsto E^{t\widetilde{\mu}}([id])$ is a path in $\T (X)$ through the basepoint $[id]$, called an {\it earthquake path}. The second author \cite{Saric} proved that there exists a neighborhood $V(\|\widetilde{\mu}\|_{Th})$ of the real axis $\mathbb{R}$ in the complex plane $\mathbb{C}$ such that the real earthquake path $t\mapsto E^{t\widetilde{\mu}}$ extends to a {\it complex earthquake}
$$
\tau\mapsto E^{\tau\widetilde{\mu}}([id])
$$
that is a holomorphic map from $V(\|\widetilde{\mu}\|_{Th})$ into $\mathcal{QF}(X)$. This implies that the real earthquake path $t\mapsto E^{t\widetilde{\mu}}$ is a real analytic path in the Teichm\"uller space $\T (X)$.

\section{The derivatives of all orders of the Liouville distributions along quake-bend paths}

Let $\widetilde{\mu}$ be the lift to $\widetilde{X}$ of a bounded measured lamination $\mu$ on $X$. 
Then $\widetilde{\mu}$ is a bounded geodesic current for $X$.
Let $\widetilde{\mu}_n^d$ be a sequence of measured laminations on $\widetilde{X}$ with discrete support that converges to $\widetilde{\mu}$ in the uniform weak* topology on $\mathcal{G}(\widetilde{X})$ as constructed in \cite{MiyachiSaric}. The support of each $\widetilde{\mu}_n^d$ is a subset of the support of $\widetilde{\mu}$ and the Thurston norm $\|\widetilde{\mu}_n^d\|_{Th}$ of the sequence $\widetilde{\mu}_n^d$ is bounded above by a constant that depends only on $\|\widetilde{\mu}\|_{Th}$. 

Let $\{K_j\}_{j=1}^{\infty}$ be an exhaustion of $G(\widetilde{X})$ by compact sets such that 
%$\partial K_j\cap |\widetilde{\mu} | =\emptyset$ 
$\widetilde{\mu} (\partial K_j) = 0$
for all $j$. For example, we can take $K_j$ to be the set of all geodesics of $\widetilde{X}\equiv\mathbb{D}$ that intersect a Euclidean disk of radius $r_j$ centered at the origin with $r_j\to 1$ as $j\to\infty$. The boundary $\partial K_j$ consists of all geodesics that intersect the circle of radius $r_j$ centered at the origin. Since we have uncountably many circles centered at the origin, there is a choice of $r_j$ such that $\widetilde{\mu}(\partial K_j)=0$.

We define $\widetilde{\mu}_{n,j}$ to be the restriction of $\widetilde{\mu}_n^d$ to $K_j$. Then $\widetilde{\mu}_{n,j}$ are measured laminations with finite support such that $\widetilde{\mu}_{n,j}\to
\widetilde{\mu}|_{K_j}$ in the weak* topology as $n\to\infty$ for each $j$. To see this, fix $\epsilon >0$. There  exists $\epsilon_j>0$ small enough, such that the $\widetilde{\mu}$-measure of the set of geodesics $E_j$ intersecting $\{ r_j-\epsilon_j\leq |z|\leq r_j+\epsilon_j\}$ is less than $\epsilon$ and that $\widetilde{\mu}_{n}^d(E_j)\to \widetilde{\mu}(E_j)<\epsilon$ (see \cite[Lemma 6]{BonahonSaric}). Let $\xi :G(\tilde{X})\to\mathbb{C}$ be an arbitrary continuous function with a compact support. Let $\xi_0 :G(\tilde{X})\to [0,1]$ be a continuous function which is constantly equal to $1$ on the set of geodesics $E_j$ that intersect $\{ |z|\leq r_j\}$ and that is equal to zero on the set of geodesics that do not intersect $\{ |z|<r_j+\epsilon_j\}$. Then we have
$$
|\int_{G(\tilde{X})}\xi d(\widetilde{\mu}_{n,j}-\widetilde{\mu}|_{K_j})|\leq |\int_{G(\tilde{X})}\xi_0\xi d(\widetilde{\mu}_{n}^d-\widetilde{\mu})|+3\epsilon
$$
for $n$ large enough. By letting $n\to\infty$ and by the convergence of $\widetilde{\mu}_n^d$ to $\widetilde{\mu}$, we conclude that
$$
\lim_{n\to\infty} |\int_{G(\tilde{X})}\xi d(\widetilde{\mu}_{n,j}-\widetilde{\mu}|_{K_j})|\leq 3\epsilon .
$$
Since $\epsilon$ was arbitrary, we conclude that $\widetilde{\mu}_{n,j}$ converge to $\widetilde{\mu}|_{K_j}$ in the weak* topology as $n\to\infty$.

By the main result in \cite{Saric2}, there exists $\delta_1=\delta_1(\|\mu\|_{\infty}) >0$ such that the complex earthquakes (or quake-bends) $E^{\tau \widetilde{\mu}}$, $E^{\tau\widetilde{\mu}|_{K_j}}$ and $E^{\tau \widetilde{\mu}_{n,j}}$ are well-defined for all $\tau\in\mathbb{C}$ with $|\tau |<\delta_1$, and they induces  holomorphic maps from $\{ |\tau |<\delta_1\}$ into the Quasi-Fuchsian space $\mathcal{QF}(\HH )\supset\mathcal{QF}(X)$. 

Given $0<\lambda\leq 1$, by \cite[Theorem 7]{DongSaric} there is $\delta_2(\lambda )>0$ such that the Liouville map
$$
\mathcal{L}:\mathcal{T} (X)\to \mathcal{G}_b(X)
$$
extends to a holomorphic map
$$
\hat{\mathcal{L}}:\mathcal{V}_{\delta_2}\to \mathcal{H}_b^{\lambda}(X)
$$
where $\mathcal{V}_{\delta_2}=\{ [\sigma ]:\sigma\in L^{\infty}(\mathbb{C})\ \mathrm{and}\ 
 \|\widetilde{\sigma}\|_{\infty}<\delta_2\}$ is an open neighborhood in $\mathcal{QF}(X)$ of the base point of the Teichm\"uller space $\T (X)$. For some $\delta >0$ small enough, we have that $[E^{\tau\widetilde{\mu}}]$, $[E^{\tau\widetilde{\mu}|_{K_j}}]$ and $[E^{\tau\widetilde{\mu}_n,j}]$ are in $\mathcal{V}_{\delta_2}$. In particular, for any $\xi\in H^{\lambda}(\widetilde{X})$ we have that
$$
\tau\mapsto \hat{\mathcal{L}}([E^{\tau\widetilde{\mu}}])(\xi ),
$$

$$
\tau\mapsto \hat{\mathcal{L}}([E^{\tau\widetilde{\mu}|_{K_j}}])(\xi ) 
$$
and
$$
\tau\mapsto \hat{\mathcal{L}}([E^{\tau\widetilde{\mu}_{n,j}}])(\xi )
$$
are holomorphic maps from $\{ |\tau |<\delta\}$ into the complex plane $\mathbb{C}$. By changing the basepoint, similar statements hold for all points of $\mathcal{T}(X)$. 

 By the construction of the quake-bends (and earthquakes) in \cite{EpsteinMarden}, we have that $E^{\tau\widetilde{\mu}_{n,j}}(x)\to E^{\tau\widetilde{\mu}|_{K_j}}(x)$ as $n\to\infty$ for each $x\in\mathbb{R}$ and the convergence is uniform for $x$ in a compact subset of $\mathbb{R}$ and all $|\tau |<\delta$.   It follows that, for each $\tau$ with $|\tau |<\delta$, 
$$
\hat{\mathcal{L}}([E^{\tau\widetilde{\mu}_{n,j}}])(\xi )\to \hat{\mathcal{L}}([E^{\tau\widetilde{\mu}|_{K_j}}])(\xi )
$$
as $n\to\infty$. Since the functions $\tau\mapsto \hat{\mathcal{L}}([E^{\tau\widetilde{\mu}_{n,j}}])(\xi )$ and $\tau\mapsto \hat{\mathcal{L}}([E^{\tau\widetilde{\mu}|_{K_j}}])(\xi )$ are holomorphic we get
$$
\frac{d^k}{d\tau^k}\hat{\mathcal{L}}([E^{\tau\widetilde{\mu}_{n,j}}])(\xi )\to \frac{d^k}{d\tau^k}\hat{\mathcal{L}}([E^{\tau\widetilde{\mu}|_{K_j}}])(\xi )
$$
as $n\to\infty$ for all $k\geq 1$ and all $j\geq 1$.

We establish a formula for the computation of the derivatives $\frac{d^k}{d\tau^k} \hat{\mathcal{L}}([E^{\tau\widetilde{\mu}}])(\xi )$ of all orders using the approximation of $\widetilde{\mu}$ with $\widetilde{\mu}_{n,j}$.  

\begin{thm}
\label{thm:derivative-limit}
Let $\mu$ be a bounded measured lamination on $X$ and $\widetilde{\mu}$ its lift to the universal covering $\widetilde{X}$. Let $0<\lambda\leq 1$ be fixed and $\delta >0$ be chosen as above depending on $\lambda$. Then, for the sequence $\{\widetilde{\mu}_{n,j}\}_{n,j=1}^{\infty}$ defined above, for any $\xi\in H^{\lambda}(\widetilde{X})$ and for any $k\geq 1$ we have
$$
\frac{d^k}{d\tau^k}\hat{\mathcal{L}}([E^{\tau\widetilde{\mu}}])(\xi )=
\lim_{j\to\infty}\Big{[}\lim_{n\to\infty} 
\frac{d^k}{d\tau^k}\hat{\mathcal{L}}([E^{\tau\widetilde{\mu}_{n,j}}])(\xi )\Big{]},
$$
for all $\tau\in\mathbb{C}$ with $|\tau |<\delta$.
\end{thm}

\begin{proof}
Since $\widetilde{\mu}_{n,j}\to\widetilde{\mu}|_{K_j}$ in the weak* topology, it follows that 
$$
\hat{\mathcal{L}}([E^{\tau\widetilde{\mu}_{n,j}}])(\xi )\to \hat{\mathcal{L}}([E^{\tau\widetilde{\mu}|_{K_j}}])(\xi )
$$
as $n\to \infty$ for all $j$ because $E^{\tau\widetilde{\mu}_{n,j}}|_{\mathbb{R}}\to E^{\tau\widetilde{\mu}|_{K_j}}|_{\mathbb{R}}$ and the definition of $\hat{\mathcal{L}}$.

The above functions are holomorphic in $\tau$ which implies that 
$$
\frac{d^k}{d\tau^k}\hat{\mathcal{L}}([E^{\tau\widetilde{\mu}_{n,j}}])(\xi )\to \frac{d^k}{d\tau^k}\hat{\mathcal{L}}([E^{\tau\widetilde{\mu}|_{K_j}}])(\xi )
$$
as $n\to \infty$ for all $j$ and $k$.

To finish the proof it remains to prove that $\frac{d^k}{d\tau^k}\hat{\mathcal{L}}([E^{\tau\widetilde{\mu}|_{K_j}}])(\xi )\to \frac{d^k}{d\tau^k}\hat{\mathcal{L}}([E^{\tau\widetilde{\mu}}])(\xi )$ as $j\to\infty$. Since $\hat{\mathcal{L}}$ is holomorphic, it is enough to prove that
\begin{equation}
\label{eq:conv-compact-general}
\lim_{j\to\infty}\hat{\mathcal{L}}([E^{\tau\widetilde{\mu}|_{K_j}}])(\xi )=\hat{\mathcal{L}}([E^{\tau\widetilde{\mu}}])(\xi )
\end{equation}
for all $|\tau |<\delta$. 

Using the partition of unity, it is enough to prove (\ref{eq:conv-compact-general}) under the assumption that the support of $\xi$ is in a box of geodesics $[a,b]\times [c,d]$, where $[a,b],[c,d]\subset\mathbb{R}$ with $[a,b]\cap [c,d]=\emptyset$. The angle metric on $[a,b]\times [c,d]$ is biLipschitz to the Euclidean metric. By the construction of the holomorphic motion $E^{\tau\widetilde{\mu}}$ restricted to the real line $\mathbb{R}$ in \cite{Saric2}, it follows that $E^{\tau\widetilde{\mu}|_{K_j}}$ converges to $E^{\tau\widetilde{\mu}}$ as $j\to\infty$ uniformly on the compact subsets of $\mathbb{R}$ for the Euclidean metric. 

Divide the box of geodesics $[a,b]\times [c,d]$ into $4^n$ sub-boxes $\{[a_{s-1},a_s]\times [c_{t-1},c_t]\}_{s,t=1}^{2^n}$ with disjoint interiors whose Liouville measures $L_{\widetilde{X}}([a_{s-1},a_s]\times [c_{t-1},c_t])$ are of the order $4^{-n}$ (see \cite{DongSaric}).  Let
$$
I_n:=\sum_{s,t=1}^{2^n}\xi (a_s,c_t) \log cr(E^{\tau\widetilde{\mu}}(a_{s-1},a_s,c_{t-1},c_t)),
$$
$$
I_n^j:=\sum_{s,t=1}^{2^n}\xi (a_s,c_t) \log cr(E^{\tau\widetilde{\mu}|_{K_j}}(a_{s-1},a_s,c_{t-1},c_t))
$$
and recall that (see \cite[Lemma 6]{DongSaric})
$$
\hat{\mathcal{L}}([E^{\tau\widetilde{\mu}}])(\xi )=I_{n_0}+\sum_{n=n_0}^{\infty} (I_{n+1}-I_n),
$$
$$
\hat{\mathcal{L}}([E^{\tau\widetilde{\mu}|_{K_j}}])(\xi )=I_{n_0}^j+\sum_{n=n_0}^{\infty} (I_{n+1}^j-I_n^j). 
$$
By \cite{DongSaric}, the sums $\sum_{n=n_0}^{\infty} |I_{n+1}-I_n|$ and $\sum_{n=n_0}^{\infty} |I_{n+1}^j-I_n^j|$ are arbitrary small for $n_0$ sufficiently large and for all $j$. The pointwise convergence of $E^{\tau\widetilde{\mu}|_{K_j}}$ to $E^{\tau\widetilde{\mu}}$ implies that $I_{n_0}^j\to I_{n_0}$. We conclude that (\ref{eq:conv-compact-general}) holds and the proof is finished.
\end{proof}

\section{A geometric formula for the first derivative of the Liouville measure}
In this section, we use Theorem \ref{thm:derivative-limit} to derive a geometric formula for the first derivative of the Liouville distributions along the earthquake and quake-bend paths. 

Let $\widetilde{\mu}_n=\sum_{i=1}^{s_n} \delta_i (\mathbf{1}_{g_i}+\mathbf{1}_{\widetilde{g}_i})$, where $\mathbf{1}_{g_i}$ is the Dirac measure on $G(\tilde{X})$ with support $\{ g_i\}$. Define
$$
h_{i,k}(\tau )=\frac{d^k}{d\tau^k}\hat{\mathcal{L} }([E^{\tau\delta_i}_{g_i}])(\xi )
$$
and note that
$$
\frac{d^k}{d\tau^k}\hat{\mathcal{L}}([E^{\tau\widetilde{\mu}_n}])(\xi )=\sum_{i=1}^{s_n}h_{i,k}(\tau ).
$$

Denote by $V_i=\frac{d}{dt}E^{t}_{g_i}|_{t=0}$ the tangent vector to the simple earthquake path $t\mapsto E^{t}_{g_i}$ at the point $t=0$. Then we have (for example, see \cite{MiyachiSaric})
$$
\frac{d}{dt}E^{t\widetilde{\mu}_n}\Big{|}_{t=0} =\sum_{i=1}^{s_n} \delta_iV_i . 
$$

We give a formula for the derivatives along a simple earthquake $t\mapsto E^t_g$ of the Liouville map evaluated at $\xi\in H(\widetilde{X})$. 

In order to find a geometric formula for the first derivative $\frac{d}{dt}{\mathcal{L}}([E^{t\widetilde{\mu}}])(\xi )$ along an earthquake $t\mapsto E^{t\widetilde{\mu}}$ for $t\in\mathbb{R}$, we first need 
the derivative along a simple earthquake path $t\mapsto E^{ta}_g$, where $a >0$ and $g\in G(\widetilde{X})$. For $h\in G(\widetilde{X})$, we define $\cos(g,h)$ to be the cosine of the angle from $g$ to $h$ if they intersect and to be zero if $h$ and $g$ do not intersect. If $g$ is the positive $y$-axis and $h$ is the geodesic from $x$ to $y$, then $\cos(g,h)=\frac{-x-y}{x-y}$. The following lemma was established by Bonahon and S\"ozen \cite{BonahonSozen}.

\begin{lem}[see \cite{BonahonSozen}]
\label{lem:simple-derivative}
Let $\xi\in H(\widetilde{X})$ and $g\in G(\widetilde{X})$ be a fixed geodesic. Then, for $t\in\mathbb{R}$ and $\omega >0$,
$$
\frac{d}{dt}\mathcal{L}([E^{t{\omega}}_g])(\xi )\Big{|}_{t=0}=\omega\int_{G(\widetilde{X})}\xi (h)\cos(g,h)dL_{\widetilde{X}}(h).
$$
\end{lem}
For a proof, see Lemma \ref{lem:simple-complex-derivative} where we prove a generalization of the above lemma.

Let $\widetilde{\mu}_{n,j}=\sum_{i=1}^{p(n,j)} \omega_i (\mathbf{1}_{g_i}+\mathbf{1}_{\widetilde{g}_i})$. Denote by $V_i=\frac{d}{dt}E^{t}_{g_i}|_{t=0}$ the tangent vector to the simple earthquake path $t\mapsto E^{t}_{g_i}$ at the point $t=0$. Then we have (for example, see \cite{MiyachiSaric})
$$
\frac{d}{dt}E^{t\widetilde{\mu}_{n,j}}\Big{|}_{t=0} =\sum_{i=1}^{p(n,j)}\omega_iV_i . 
$$
Since $d\mathcal{L}:T_{[id]}\T (\widetilde{X})\to\mathcal{H}_b(\widetilde{X})$ is linear, by Lemma \ref{lem:simple-derivative}, we get that
\begin{equation}
\label{eq:finite-earthquake-derivative}
\begin{split}
d\mathcal{L}\Big{(}\frac{d}{dt}E^{t\widetilde{\mu}_{n,j}}|_{t=0}\Big{)}(\xi )=\sum_{i=1}^{p(n,j)}\omega_i \int_{G(\widetilde{X})}\xi (h)\cos(g_i,h)dL_{\widetilde{X}}(h)\\ =\int_{G(\widetilde{X})}\int_{G(\widetilde{X})}\xi (h)\cos(g,h)dL_{\widetilde{X}}(h) d\widetilde{\mu}_{n,j}(g).
\end{split}
\end{equation}

\begin{rem}\label{rem:continuity}
	We claim that $g\mapsto \int_{G(\widetilde{X})}\xi (h)\cos(g,h)dL_{\widetilde{X}}(h)$ is a continuous function in $g$. It is enough to assume that $Supp(\xi)=[a,b] \times [c,d]$ is a box of geodesics. To see this, let $g, g' \in G(\widetilde X)$ where $g \neq g'$ and suppose that $g \in G(\widetilde{X})$ has endpoints $s$ and $t$. We define the closure of $\delta$-neighborhood of $g$ to be the box $\overline {B_{\delta}(g)}=[s-\delta, s+\delta] \times [t-\delta, t+\delta]$. We say that $g' \to g$ if $g' \in \overline {B_{\delta}(g)}$ as $\delta \to 0$. For all $h \in Supp(\xi)$ with no endpoints in $\overline {B_{\delta}(g)}$, we have $\int_{G(\widetilde{X})} |\xi(h)| |\cos(g',h)-\cos(g,h)| dL_{\widetilde{X}}(h) \to 0$ as $g' \to g$ since $|\cos(g',h) - \cos(g,h)| \to 0$ uniformly as $g' \to g$. For all $h \in Supp(\xi)$ with at least one endpoint in $\overline {B_{\delta}(g)}$, we also have $\int_{G(\widetilde{X})} |\xi(h)| |\cos(g',h)-\cos(g,h)| dL_{\widetilde{X}}(h) \to 0$ as $g' \to g$ since the Liouville measure $L([s-\delta, s+\delta] \times [c, d])=L([a, b] \times [t-\delta, t+\delta]) \to 0$ as $\delta \to 0$.
\end{rem}

Since $\widetilde{\mu}_{n,j}$ converges in the weak* topology to $\widetilde{\mu}|_{K_j}$ as $n\to\infty$, it follows that
$$
\lim_{n\to\infty} \int_{G(\widetilde{X})}\int_{G(\widetilde{X})}\xi (h)\cos(g,h)dL_{\widetilde{X}}(h) d\widetilde{\mu}_{n,j}(g)= \int_{G(\widetilde{X})}\int_{G(\widetilde{X})}\xi (h)\cos(g,h)dL_{\widetilde{X}}(h) d\widetilde{\mu}|_{K_j}(g).
$$
By Theorem \ref{thm:derivative-limit}, we have that
$$
\frac{d}{dt}{\mathcal{L}}([E^{t\widetilde{\mu}}])(\xi )\Big{|}_{t=0}=\lim_{j\to\infty} \int_{G(\widetilde{X})}\int_{G(\widetilde{X})}\xi (h)\cos(g,h)dL_{\widetilde{X}}(h) d\widetilde{\mu}|_{K_j}(g).
$$
We point out that even though the limit on the right-hand side of the above equation exists and that $\widetilde{\mu}|_{K_j}\to\widetilde{\mu}$ in the weak* topology as $j\to\infty$, we are not guaranteed that the limit is the Lebesgue integral with respect to $\widetilde{\mu}$ because $\int_{G(\widetilde{X})}\xi (h)\cos(g,h)dL_{\widetilde{X}}(h)$ is neither a positive function nor of compact support.  In order to prove the convergence toward the integral with respect to the measure $\widetilde{\mu}$, we need the following lemma proved by Bonahon and S\" ozen \cite{BonahonSozen}. We include a proof for the reader's convenience.

\begin{lem}
\label{lem:bound-deriv} Fix an isometric identification of $\widetilde{X}$ with the unit disk $\mathbb{D}$. 
Let $\xi :G(\mathbb{D})\to\mathbb{R}$ be a $\lambda$-H\"older continuous function whose support is in a box of geodesics $[a,b]\times [c,d]\subset S^1\times S^1\setminus\mathrm{diag}$. Let $g\in G(\mathbb{D})$ be a geodesic with at least one endpoint in $[a,b]\cup [c,d]$. Then
$$
\Big{|}\int_{G(\mathbb{D})}\xi (h)\cos(g,h)dL_{\mathbb{D}}(h)\Big{|} \leq \int_{G(\mathbb{D})}\Big{|}\xi (h)\cos(g,h)\Big{|}dL_{\mathbb{D}}(h) \leq C e^{-(1+\lambda )d_g}
$$
where $d_g\geq 0$ is the hyperbolic distance between $0\in\mathbb{D}$ and the geodesic $g$.
\end{lem}
\begin{proof}
	To show that $\Big{|}\int_{G(\mathbb{D})}\xi (h)\cos(g,h)dL_{\mathbb{D}}(h)\Big{|}$ is bounded, it is enough to show that for fixed $\xi$ the integral
	$$I(y)=2 \int_s^t \xi(h(x,y)) \cos(g, h(x,y)) \frac{dx}{|x-y|^2}$$ is bounded for every $y \in (a, b)$. The factor $2$ is here to avoid dragging cumbersome constants in the computation below.
	
	\begin{figure}[h!]
		\centerline{
			\includegraphics[width=0.8\textwidth]{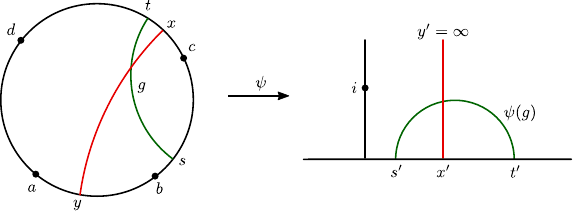}
		}
		\caption{\small{Earthquake along $g$ and M\"obius map $\psi$.}}\label{figure:Sine_Estimate}
	\end{figure}
	
	For ease of computation, we switch to the upper half-plane $\HH$ such that origin is sent to $i$ and $y$ is sent to $\infty$ by the M\"obius map $\psi: \D \to \HH$ given by $\psi(z)=\frac{i(y+z)}{y-z}$. A simple calculation shows that the infinitesimal form $\frac{dx}{|x-y|^2}$ is sent to $\frac{1}{2}dx'$.
	
	Then we have
	\begin{align*}
		I(\infty)
		&= \int_{s'}^{t'} \xi \circ \psi^{-1}(h(x', \infty)) \cos(\psi(g), h(x',\infty))dx'.
	\end{align*}
	
	To simplify the notation, let $\xi'(x')=\xi \circ \psi^{-1}(h(x', \infty))$. Note that there exists some $r' \in (s', t')$ such that $\xi'(r')=0$. That is, $(r', \infty) \notin Supp(\xi \circ \psi^{-1})$, see Figure \ref{figure:Sine_Estimate}. 
	
	Estimated from above, we have
	\begin{align*}
		|I(\infty)| 
		&= \bigg|\int_{s'}^{t'}  \xi'(x') \cos(\psi(g), h(x',\infty))dx'\bigg| \\
		&\leq \big \Vert \xi'(x') \big \Vert_{\infty} |t'-s'| 
	\end{align*}
	
	Note that $\xi'$ is H\"older continuous with H\"older exponent $\lambda$ since $\xi \circ \psi^{-1}$ is.
	By definition of $\lambda$-H\"older continuity, we have $\displaystyle \frac{\big |\xi'(x') - \xi'(r') \big |}{|x'-r'|^{\lambda}} \leq \Vert \xi' \Vert_{\lambda}$, then
	$$|\xi'(x')| = \big |\xi'(x')-\xi'(r') \big | \leq \Vert \xi' \Vert_{\lambda} \ |x'-r'|^{\lambda}$$ for all $x'$
	which implies that $\big \Vert \xi'(x') \big \Vert_{\infty} \leq \Vert \xi' \Vert_{\lambda} \ |x'-r'|^{\lambda}$. Hence, 
	\begin{align*}
		\big | I(\infty) \big |
		&\leq \Vert \xi' \Vert_{\lambda} \ |x'-r'|^{\lambda} |t'-s'| \\
		&\leq \Vert \xi' \Vert_{\lambda} \ |t'-s'|^{1+\lambda}.
	\end{align*} 
	
	\begin{figure}[h!]
		\centerline{
			\includegraphics[width=0.5\textwidth]{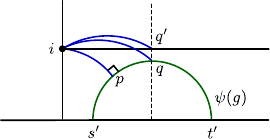}
		}
		\caption{\small{Orthogonal projections of $i$ and $\infty$ to $\psi(g)$.}}\label{figure:Interval_Estimate}
	\end{figure}
	
	And the rest of the proof follows from \cite{BonahonSozen} where it gives the estimate $|t'-s'| \leq C \cdot e^{-d_g}$. To see this, we suppose that $I(\infty) \neq 0$ and $d_g=d_{\D}(0,g)=d_{\HH}(i, \psi(g))=d_{\HH}(i, p) \geq R$ for some constant $R$. By our assumption, $s'$ and $t'$ stay in a closed interval $[-R', R']$ where $R'$ depends only on $R$, see Figure \ref{figure:Interval_Estimate}. Consider a point $q$ which is the orthogonal projection of the point $\infty$ to $\psi(g)$, namely $q=\frac{t'+s'}{2}+\frac{t'-s'}{2}i$. Also consider the horocycle centered at $\infty$ passing through $i$, namely the Euclidean horizontal line passing through $i$. Let $q'$ be the point of this horocycle which lies on the same vertical line as $q$, namely $q'=\frac{t'+s'}{2}+i$. Note that the piece of horocycle joining $i$ and $q'$ has length $\vert \frac{t'+s'}{2} \vert \leq R'$. It follows that $d_{\HH}(i, q') \leq R'$ since $(i, q')$ is a geodesic. 
	%Also since $p$ is the orthogonal projection of $i$ to $\psi(g)$ and the orthogonal projection map is distance non-increasing, $d_{\HH}(p, q) \leq R'$. 
	Thus,
	%	\begin{align*}
		%		d_{\HH}(i, p)=d_{\HH}(i, \psi(g))
		%		&\leq d_{\HH}(i, q) + d_{\HH}(q, p)\\
		%		&\leq d_{\HH}(i, q') + d_{\HH}(q', q) + d_{\HH}(q, p)\\
		%		&\leq 2R' + d_{\HH}(q', q)\\
		%		&= 2R' + \log{\frac{t'-s'}{2}}.
		%	\end{align*} 
	\begin{align*}
		d_{\HH}(i, p)=d_{\HH}(i, \psi(g))
		&\leq d_{\HH}(i, q)\\
		&\leq d_{\HH}(i, q') + d_{\HH}(q', q)\\
		&\leq R' + d_{\HH}(q', q)\\
		&= R' + \log{\frac{t'-s'}{2}}.
	\end{align*}Since $\frac{t'-s'}{2}\leq R'$, we have 
	\begin{align*}
		e^{-d_{\HH}(i, p)} 
		&\geq e^{-R'-\log{\frac{t'-s'}{2}}}\\
		&= e^{-R'} \cdot \frac{2}{t'-s'}\\
		&\geq \frac{e^{-R'}}{2(R')^2} \cdot (t'-s')
	\end{align*} This proves $|t'-s'| \leq C \cdot e^{-d_g}$.
\end{proof}

\begin{rem}\label{rem:bound-deriv-sin}
We similarly denote by $\sin(g,h)$ to be the sine of the angle from $g$ to $h$ if they intersect and to be zero if $h$ and $g$ do not intersect. The following estimate (similar to Lemma \ref{lem:bound-deriv}) can be made
	$$
	\Big{|}\int_{G(\mathbb{D})}\xi (h)\sin(g,h)dL_{\mathbb{D}}(h)\Big{|}\leq \int_{G(\mathbb{D})}\Big{|}\xi (h)\sin(g,h)\Big{|}dL_{\mathbb{D}}(h) \leq C e^{-(1+\lambda )d_g}
	$$
	where $d_g\geq 0$ is the hyperbolic distance between $0\in\mathbb{D}$ and the geodesic $g$.
\end{rem} And now we prove
\begin{thm}
\label{thm:earthquake-derivative}
Let $\mu$ be a bounded measured lamination on a conformally hyperbolic Riemann surface $X$ and let $\widetilde{\mu}$ be its lift to the universal covering $\widetilde{X}$. Then
\begin{equation}
\label{eq:derivative_main}
\frac{d}{dt}{\mathcal{L}}([E^{t\widetilde{\mu}}])(\xi )\Big{|}_{t=0}=\int_{G(\widetilde{X})} \int_{G(\widetilde{X})} \xi (h)\cos(g,h)dL_{\widetilde{X}}(h)d\widetilde{\mu}(g)
\end{equation}
\end{thm}

\begin{proof} In order to finish the proof it remains to prove that 
$$\lim_{j\to\infty} \int_{G(\widetilde{X})}\int_{G(\widetilde{X})}\xi (h)\cos(g,h)dL_{\widetilde{X}}(h) d\widetilde{\mu}|_{K_j}(g)$$ is equal to the double integral in the statement of the theorem. By the dominated convergence theorem, it is enough to show that 
$$
\int_{G(\widetilde{X})}\Big{|}\int_{G(\widetilde{X})}\xi (h)\cos(g,h)dL_{\widetilde{X}}(h)\Big{|}d\widetilde{\mu}(g)<\infty .
$$

We identify $\widetilde{X}$ with $\mathbb{D}$ by an isometry. Let $C_n=\{ z\in\mathbb{D}:n<\rho_{\mathbb{D}}(0,z)\leq n+1\}$ be the half-closed annulus around $0$ with inner radius $n$ and outer radius $n+1$. For each $n\geq 2$, let  $\mathcal{F}_n$ be the family of geodesics of the support $|\widetilde\mu |$ of $\widetilde\mu$ that intersect $C_n$ but do not intersect $C_{n-1}$. We partition $\mathcal{F}_n$ into finitely many subfamilies, see Figure \ref{figure:Compact_Subfamilies}. Namely, if $g,g'\in\mathcal{F}_n$ and either $g$ separates $g'$ and $0$, or $g'$ separates $g$ and $0$ then $g$ and $g'$ belongs to the same subfamily. Each subfamily has the unique geodesic that is farthest away from $0$. 

\begin{figure}[h!]
	\centerline{
		\includegraphics[width=0.3\textwidth]{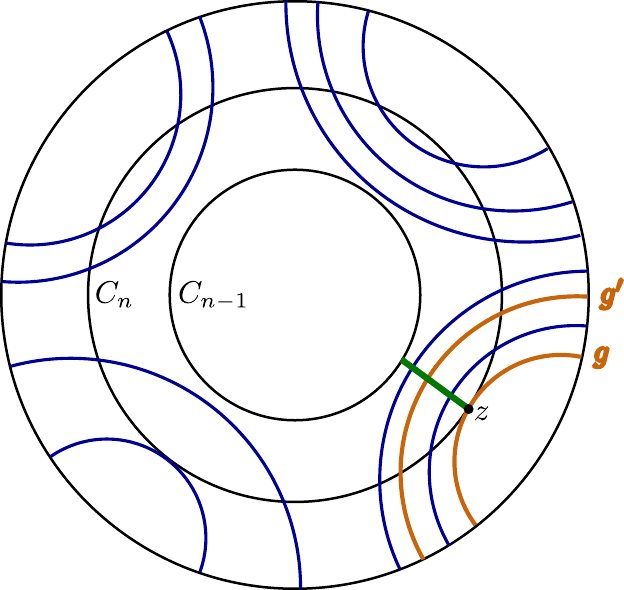}
	}
	\caption{\small{Construction of $\mathcal{F}_n$ and corresponding subfamilies.}}\label{figure:Compact_Subfamilies}
\end{figure}

Let $g$ be one geodesic in a subfamily of $\mathcal{F}_n$ that is farthest away from $0$ and let $z\in g$ be the closest point of $g$ to $0$. Then the shortest distance $\rho (0,z)$  from $g$ to $0$ is less than or equal to $n+1$. This is equivalent to $\log \frac{1+|z|}{1-|z|} \leq n+1$ which implies $1-|z| \geq C_1e^{-n}$ for some universal constant  $C_1>0$. Therefore the Euclidean circle which contains the geodesic $g$ has the Euclidean radius $r\geq C_1e^{-n}$. Therefore the arc length of the unit circle $S^1$ cut out by the geodesic $g$ is at least $C_1e^{-n}$ by the inequality $\tan^{-1}(x) > \frac{1}{2}x$ for $0<x\leq 1$ and some elementary Euclidean considerations. We conclude that the number of subfamilies of $\mathcal{F}_n$ is at most $\frac{2\pi}{C_1}e^n$.

To simplify the notation, set $I(g)=\Big{|}\int_{G(\widetilde{X})}\xi (h)\cos(g,h)dL_{\widetilde{X}}(h)\Big{|}$. By Lemma \ref{lem:bound-deriv}, we have that $I(g)\leq C e^{-(1+\lambda )n}$. The total measure of the geodesics in each subfamily $\mathcal{F}_n^i$ of $\mathcal{F}_n$ is at most $\|\widetilde\mu\|_{Th}$ for $\widetilde\mu$ because each subfamily intersects an arc of the radius of hyperbolic length $1$. It follows that
$\int_{\mathcal{F}_n^i}I(g)d\widetilde\mu (g)\leq Ce^{-(1+\lambda )n}$ for each subfamily $\mathcal{F}_n^i$ and the total integral $\int_{\mathcal{F}_n}I(g)d\widetilde\mu(g)$ is bounded by $C_2e^{-\lambda n}$ for some universal constant $C_2>0$. Since $\sum_{n=1}^{\infty} e^{-\lambda n}<\infty$ 
%we can find $n_0$ such that  $\sum_{n=n_0}^{\infty} e^{-\lambda n}<\epsilon$ and $j_{\epsilon}>0$ such that $\mathcal{F}_n\cap K_{j_{\epsilon}}=\emptyset$ for $n\leq n_0$  and 
, the theorem is proved.
\end{proof}

By changing the basepoint of $\mathcal{T}(X)$, Theorem \ref{thm:earthquake-derivative} gives
\begin{cor}
	\label{thm:earthquake-derivative_for_ any_t}
	Let $\mu$ be a bounded measured lamination on a conformally hyperbolic Riemann surface $X$ and let $\widetilde{\mu}$ be its lift to the universal covering $\widetilde{X}$. Then
	\begin{equation}
		\label{eq:derivative_main_for_any_t}
		\frac{d}{dt}{\mathcal{L}}([E^{t\widetilde{\mu}}])(\xi )=\int_{G(\widetilde{X})} \int_{G(\widetilde{X})} \xi (h)\cos(E^{t\widetilde{\mu}}(g),E^{t\widetilde{\mu}}(h))dL_{[E^{t\widetilde{\mu}}]}(h)d\widetilde{\mu}(g)
	\end{equation}
\end{cor}

In order to extend the above formula to quake-bends, we first need an extension of Lemma \ref{lem:simple-derivative}. We begin with the definition of signed complex distance in $\HH^3$, see \cite{Series}.
Let $\alpha$ be any oriented line in $\HH^3$, and let $P_1, P_2 \in \alpha$. Let $d(P_1,P_2)\geq 0$ denote the hyperbolic distance between $P_1$ and $P_2$. We define the signed real hyperbolic distance $\delta_{\alpha}(P_1,P_2)$ as $d(P_1,P_2)$ if the orientation of the arc from $P_1$ to $P_2$ coincides with that of $\alpha$ and $-d(P_1,P_2)$ otherwise. Now let $L_1, L_2 \subset \HH^3$ be oriented lines with distinct endpoints on the Riemann sphere $\hat \C$, with oriented common perpendicular $\alpha$ meeting $L_1, L_2$ in points $Q_1, Q_2$ respectively, where if $L_1, L_2$ intersect we take $Q_1=Q_2$. Let $\bm{\mathrm{v_i}}$ be tangent vectors to the positive directions of $L_i$ at $Q_i$, $i=1,2$. Let $\Pi$ be the hyperbolic plane through $Q_2$ orthogonal to $\alpha$ and let $\bm{\mathrm{w_1}}$ denote the parallel translate of $\bm{\mathrm{v_1}}$ along $\alpha$ to $Q_2$. Let $\bm{\mathrm{n}}$ be a unit vector at $Q_2$ pointing in the positive direction of $\alpha$. The signed complex distance between $L_1$ and $L_2$ is $$\bm{d}(L_1,L_2)=\bm{\delta_n}(L_1,L_2)=\bm{\delta}_{\alpha}(L_1,L_2)=\delta_{\alpha}(Q_1,Q_2)+i\theta$$ where $\theta$, measured modulo $2\pi\Z$, is the angle between $\bm{\mathrm{w_1}}$ and $\bm{\mathrm{v_2}}$ measured anticlockwise in the plane spanned by $\bm{\mathrm{w_1}}, \bm{\mathrm{v_2}}$ and oriented by $\bm{\mathrm{n}}$.

\begin{lem}
\label{lem:complex-distance}
Let $x,y\in \R$ and $\omega>0, \tau \in \C$ with $|\tau|<\delta$. Let $g$ be the geodesic with endpoints oriented from $0$ to $\infty$ and $h$ be the geodesic with endpoints oriented from $xe^{\tau\omega}$ to $y$. Then
$$\cosh(\bm{d}(g,h))=\frac{-xe^{\tau\omega}-y}{xe^{\tau\omega}-y}=1-\frac{2}{cr(y, 0, xe^{\tau\omega}, \infty)}$$ where $\bm{d}$ is the complex distance.
\end{lem}

\begin{proof}
For $|\tau|<\delta$, $xe^{\tau\omega}$ is in a small neighborhood of $x$.
A simple computation shows that there is a unique common perpendicular geodesic $l$ with endpoints $-\sqrt{xye^{\tau\omega}}, \sqrt{xye^{\tau\omega}}$ intersecting $g, h$ at $p, q$, respectively. 

Define a M\"obius map $f(z)=\frac{z-\sqrt{xye^{\tau\omega}}}{z+ \sqrt{xye^{\tau\omega}}}$, which sends $l$ to the geodesic $l'=(0, \infty)$, $g$ to the geodesic $g'=(-1, 1)$ and $h$ to the geodesic $h'=(f(xe^{\tau\omega}), f(y))$, see Figure \ref{figure:H3_Distance}. Note that $l'$ intersects $g', h'$ at $p'=f(p), q'=f(q)$, respectively. A direction computation shows that $f(y)=\frac{y-\sqrt{xye^{\tau\omega}}}{y+ \sqrt{xye^{\tau\omega}}}=-f(xe^{\tau\omega})$, this implies that they are diametrically opposite of the origin.\\

\begin{figure}[h!]
	\centerline{
		\includegraphics[width=0.8\textwidth]{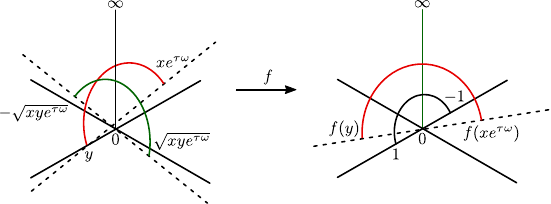}
	}
	\caption{\small{Hyperbolic distance in upper-half space.}}\label{figure:H3_Distance}
\end{figure}

Now we have the hyperbolic distance,\\

$d(g,h)=d(g',h')=d(p',q')=
\begin{cases}
\displaystyle \log|f(y)| \quad \text{if } |f(y)|\geq 1\\
\\

\displaystyle \log\frac{1}{|f(y)|} \quad \text{if } |f(y)|<1.
\end{cases}$\\
\\

By our definition of signed complex distance in $\HH^3$,
\begin{align*}
\bm{d}(g,h)=\bm{d}(g',h')
&=\delta_{l'}(p',q')+i\theta\\
&=\log|f(y)|+i(arg(f(y)))\\
&=\log(f(y)).
\end{align*} Finally, 
\begin{align*}
\cosh(\bm{d}(g,h))=\frac{e^{\log(f(y))}+e^{-\log(f(y))}}{2}
&=\frac{f(y)+\frac{1}{f(y)}}{2}\\
&=\frac{1}{2}\bigg(\frac{y-\sqrt{xye^{\tau\omega}}}{y+ \sqrt{xye^{\tau\omega}}}+\frac{y+\sqrt{xye^{\tau\omega}}}{y-\sqrt{xye^{\tau\omega}}}\bigg)\\
&=\frac{-xe^{\tau\omega}-y}{xe^{\tau\omega}-y}
\end{align*} and the second equality is straightforward.
\end{proof}

The following result will be useful in our paper.
\begin{lem}[see \cite{Dong}]
	\label{lem:leibniz_rule}
	Let $X$ be an open subset of $\mathbb{C}$ and $\Omega$ be a measure space. Suppose that $f: X \times \Omega \to \mathbb{C}$ satisfies the following conditions:
	\begin{enumerate}
		\item[(i)] $f(\tau, \omega)$ is a Lebesgue integrable function of $\omega$ for each $\tau \in X$.
		
		\item[(ii)] For almost all $\omega \in \Omega$, the derivative $\frac{\partial f(\tau, \omega)}{\partial \tau}$ exists for all $\tau \in X$.
		
		\item[(iii)] There is an integrable function $\Theta: \Omega \to \mathbb{C}$ such that $\Big| \frac{\partial f(\tau, \omega)}{\partial \tau} \big| \leq \Theta$ for all $\tau \in X$.
	\end{enumerate}
	Then for all $\tau \in X$, $\frac{d}{d\tau} \int_{\Omega} f(\tau, \omega) d\omega = \int_{\Omega} \frac{\partial}{\partial \tau} f(\tau, \omega) d\omega$.
\end{lem}

We first find the derivative along a simple quake-bend (which generalizes Lemma \ref{lem:simple-derivative}):
\begin{lem}
\label{lem:simple-complex-derivative}
Let $\xi\in H(\widetilde{X})$ and $g\in G(\widetilde{X})$ be a fixed geodesic. Let $\delta >0$ be the radius for which $\hat{\mathcal{L}}([E^{\tau \omega}_g])(\xi )$ is defined. Then, for $\tau\in\mathbb{C}$ with $|\tau |<\delta$ and $\omega>0$,
$$
\frac{d}{d\tau}\hat{\mathcal{L}}([E^{\tau \omega}_g])(\xi )=\omega\int_{G(\widetilde{X})}\xi (h)\cosh(\bm{d}(E^{\tau \omega}_g(g),E^{\tau \omega}_g(h)))dL_{[E^{\tau\omega}_g]}(h),
$$
where $E^{\tau \omega}_g(g)=g$ and $\bm d(E^{\tau \omega}_g(g),E^{\tau \omega}_g(h))$ is the complex distance between $E^{\tau \omega}_g(g)$ and $E^{\tau \omega}_g(h)$.
\end{lem}

\begin{proof}
We identify $\widetilde{X}$ with the upper half-plane $\mathbb{H}$ such that $g$ is identified with the positive $y$-axis. Without loss of generality, we assume that $\xi$ has support in a box of geodesics and let $[a,b]\times [c,d]\subset G(\mathbb{H})$ be the sub-box of the support of $\xi$ such that each geodesic of its interior intersects $g$. 
A computation for $t$ real gives that 
\begin{flalign}\label{first-derivative-real}
	\frac{d}{dt} {\mathcal{L}}([E^{t \omega}_g])(\xi )
	&=\frac{d}{dt} \iint_{[a,b]\times [c,d]} \xi \circ (E_g^{t\omega})^{-1}(h(x, y)) \frac{dxdy}{(x-y)^2}\\
	&=\frac{d}{dt} \iint_{[a,b]\times [c,d]} \xi (h(e^{-t\omega}x, y)) \frac{dxdy}{(x-y)^2}\nonumber\\
	&=\omega\iint_{[a,b]\times [c,d]} \xi (h(x,y))\frac{-xe^{t\omega}-y}{xe^{t\omega}-y}\frac{e^{t\omega}dxdy}{(xe^{t\omega}-y)^2}.\nonumber
\end{flalign}
Since $\tau\mapsto \hat{\mathcal{L}}([E^{\tau \omega}_g])(\xi )$ is holomorphic in $\tau$ and by Lemma \ref{lem:leibniz_rule} the integral
$$\omega\iint_{[a,b]\times [c,d]} \xi (h(x,y))\frac{-xe^{\tau\omega}-y}{xe^{\tau\omega}-y}\frac{e^{\tau\omega}dxdy}{(xe^{\tau\omega}-y)^2}$$ is also holomorphic in $\tau$ , hence by the uniqueness of holomorphic maps we have
\begin{equation}\label{first-derivative-complex}
\frac{d}{d\tau} \hat{\mathcal{L}}([E^{\tau \omega}_g])(\xi )=\omega\iint_{[a,b]\times [c,d]} \xi (h(x,y)) \frac{-xe^{\tau\omega}-y}{xe^{\tau\omega}-y}\frac{e^{\tau\omega}dxdy}{(xe^{\tau\omega}-y)^2}.
\end{equation}
Note that the density $$\frac{e^{\tau\omega}dxdy}{(xe^{\tau\omega}-y)^2}$$ defines a countably additive complex measure on $[a,b]\times [c,d]$ of finite variation.\\
The result follows immediately from Lemma \ref{lem:complex-distance}.
\end{proof}

By Lemma \ref{lem:simple-complex-derivative}, we have
$$
\frac{d}{d\tau}\hat{\mathcal{L}}([E^{\tau\widetilde{\mu}_{n,j}}])(\xi )=\sum_{i=1}^{p(n,j)}\omega_i\int_{G(\widetilde{X})}\xi (h)\cosh \bm{d}(E^{\tau\widetilde{\mu}_{n,j}}(g_i),E^{\tau\widetilde{\mu}_{n,j}}(h))dL_{[E^{\tau\widetilde{\mu}_{n,j}}]}(h),
$$
where $\widetilde{\mu}_{n,j}=\sum_{i=1}^{p(n,j)} \omega_i (\mathbf{1}_{g_i}+\mathbf{1}_{\widetilde{g}_i})$. 
The above formula can be written as
$$
\frac{d}{d\tau}\hat{\mathcal{L}}([E^{\tau\widetilde{\mu}_{n,j}}])(\xi )=\int_{G(\widetilde{X})}\int_{G(\widetilde{X})}\xi (h)\cosh \bm{d}(E^{\tau\widetilde{\mu}_{n,j}}(g),E^{\tau\widetilde{\mu}_{n,j}}(h))dL_{[E^{\tau\widetilde{\mu}_{n,j}}]}(h)d\widetilde{\mu}_{n,j}(g).
$$
By Remark \ref{rem:continuity}, $g \mapsto \int_{G(\widetilde{X})}\xi (h)\cosh \bm{d}(E^{\tau\widetilde{\mu}_{n,j}}(g),E^{\tau\widetilde{\mu}_{n,j}}(h))dL_{[E^{\tau\widetilde{\mu}_{n,j}}]}(h)$ is a continuous function in $g$. Since $\widetilde{\mu}_{n,j}$ converges in the weak* topology to $\widetilde{\mu}|_{K_j}$ as $n\to\infty$, it follows that
$$
\lim_{n\to\infty} 
\frac{d}{d\tau}\hat{\mathcal{L}}([E^{\tau\widetilde{\mu}_{n,j}}])(\xi )=\int_{G(\widetilde{X})}\int_{G(\widetilde{X})}\xi (h)\cosh \bm{d}(E^{\tau\widetilde{\mu}|_{K_j}}(g),E^{\tau\widetilde{\mu}|_{K_j}}(h))dL_{[E^{\tau\widetilde{\mu}|_{K_j}}]}(h)d{\widetilde\mu}|_{K_j}(g).
$$ By Theorem \ref{thm:derivative-limit}, we obtain a formula for the derivative of the Liouville H\"older distributions at a point of the Quasi-Fuachsian space corresponding to a quake-bend with small imaginary parts. Note that the convergence of the formula is conditional (given by a limit) unlike the convergence at Teichm\"uller space where the convergence is absolute (given by a Lebesgue integral in the measure $\tilde{\mu}$). One would expect that the convergence is conditional due to the fact that the quasicircles do not give (bounded variation complex) measures but rather only  H\"older distributions as in \cite{DongSaric}.

\begin{thm}
	\label{thm:quake-der}
	Let $\xi\in H(\tilde{X})$ and let $\delta >0$ be the radius for which $\hat{\mathcal{L}}([E^{\tau \tilde{\mu}}])(\xi )$ is defined. Then, for $\tau\in\mathbb{C}$ with $|\tau |<\delta$,
	$$
	\frac{d}{d\tau} \hat{\mathcal{L}}([E^{\tau \tilde{\mu}}])(\xi )=\lim_{j\to\infty}\int_{G(\tilde{X})}\int_{G(\tilde{X})}\xi (h)\cosh \bm{d}(E^{\tau\widetilde{\mu}|_{K_j}}(g),E^{\tau\widetilde{\mu}|_{K_j}}(h))dL_{[E^{\tau\widetilde{\mu}|_{K_j}}]}(h)d\tilde{\mu}|_{K_j}(g)
	$$ where $\widetilde{\mu}|_{K_j}\to\widetilde{\mu}$ in the weak* topology as $j\to\infty$ and $\bm{d}$ is the complex distance.
\end{thm}

\section{A geometric formula for the second derivative of the Liouville measure}

By taking derivatives of the equations \eqref{first-derivative-real} and  \eqref{first-derivative-complex}, we obtain
\begin{cor}
	\label{cor:simple-second-derivative}
	Let $\xi\in H(\widetilde{X})$ and $g\in G(\widetilde{X})$ be a fixed geodesic. Then, for $t\in\mathbb{R}$ and $\omega >0$,
	$$\frac{d^2}{dt^2}{\mathcal{L}}([E^{t{\omega}}_g])(\xi ) \Big{|}_{t=0} = {\omega^2} \int_{G(\widetilde{X})}\xi (h)[\cos^2 (g, h) -  \frac{1}{2} \sin^2 (g, h)]dL_{\widetilde{X}}(h).$$
\end{cor}	
\begin{cor}
\label{cor:simple-complex-second-derivative}
	Let $\xi \in H(\widetilde X)$ and $g \in G(\widetilde X)$ be a fixed geodesic. Let $\delta>0$ be the radius for which the quake-bend is defined. Then, for $\tau \in \C$ with $|\tau|<\delta$ and $\omega>0$, 
	
	$$\frac{d^2}{d\tau^2}\hat{\mathcal{L}}([E^{\tau{\omega}}_g])(\xi ) = {\omega^2} \int_{G(\widetilde{X})}\xi (h)[\cosh^2 \bm{d}(E^{\tau\omega}_g(g),E^{\tau\omega}_g(h)) -  \frac{1}{2} \sinh^2 \bm{d}(E^{\tau\omega}_g(g), E^{\tau\omega}_g(h))]dL_{[E^{\tau\omega}_g]}(h).$$ where $E^{\tau\omega}_g(g)=g$ and $\bm{d}$ is the complex distance.
\end{cor}

However, the above two corollaries do not provide the second derivatives in the case when the measures are finite combinations of Dirac measures due to the interactions of different terms.
We would like to find the formula for the second derivative of the Liouville measure along an elementary earthquake, i.e. an earthquake whose measure is a finite sum of Dirac measures. Since the full computations are much more involved, we point out that the second derivative captures all pairwise interactions of the supporting geodesics in $|\widetilde{\mu}_{n,j}|$.

To see this, let $E^{t\sigma} = E_{g_i}^{t\omega_i} \circ E_{g_k}^{t\omega_k}$ be an elementary earthquake with measured lamination $\sigma$ supported on two disjoint geodesics $g_i$ and $g_k$ as defined in section 3. Suppose that $g_i = (0, \infty)$, $g_k = (\alpha, \beta)$ and $h(x, y)$ is a geodesic intersecting both $g_i$ and $g_k$. Define a M\"obius map $\varphi(z) = \frac{\alpha - z}{z - \beta}$ which sends $\alpha$ to 0 and $\beta$ to $\infty$. By definition,
$E_{g_i}^{t\omega_i} = 
\begin{cases}
	x \quad \quad \text{ if } x \leq 0\\
	e^{t\omega_i}x \ \ \text{ if } x>0.
\end{cases}$ 

For $t \in \R$, we have 
\begin{flalign*}
	\frac{d}{dt}\mathcal{L}([E^{t{\sigma}}_g])(\xi) \bigg\vert_{t=0}
	&=\frac{d}{dt} \int_{G(\widetilde{X})} \xi \circ (E^{t\sigma})^{-1}(h(x, y)) \frac{dxdy}{(x-y)^2}\bigg\vert_{t=0}\\
	&=\frac{d}{dt} \int_{G(\widetilde{X})} \xi \circ \varphi^{-1} \circ \varphi \circ (E_{g_k}^{t\omega_k})^{-1} \circ \varphi^{-1} \circ \varphi(h(e^{-t\omega_i}x, y))\frac{dxdy}{(x-y)^2}\bigg\vert_{t=0}\\
	&=\frac{d}{dt} \int_{G(\widetilde{X})} \xi \circ \varphi^{-1} \circ (E_{\varphi(g_k)}^{t\omega_k})^{-1} \circ \varphi(h(e^{-t\omega_i}x, y))\frac{dxdy}{(x-y)^2}\bigg\vert_{t=0}\\
	&=\int_{G(\widetilde{X})} \xi(h) \bigg[\omega_i \cos(g_i, h) + \omega_k \cos(g_k, h)\bigg] dL_{\widetilde X}(h)
\end{flalign*} and
\begin{flalign*}
	\frac{d^2}{dt^2}\mathcal{L}([E^{t{\sigma}}_g])(\xi) \bigg\vert_{t=0}
	&=\int_{G(\widetilde{X})} \xi(h)\\
	&\qquad\qquad \bigg[\omega_i^2 \big[\cos^2(g_i, h) - \frac{1}{2} \sin^2(g_i, h)\big] + \omega_k^2 \big[\cos^2(g_k, h) - \frac{1}{2} \sin^2(g_k, h)\big]\\
	&\qquad\qquad + \omega_i \omega_k \big[\cos(g_i, h)\cos(g_k, h) - \frac{1}{2} \sin(g_i, h)\sin(g_k, h) e^{-d_h}\big]\bigg] dL_{\widetilde X}(h)
\end{flalign*} where $d_h$ is the hyperbolic distance along $h$ from $g_i \cap h$ to $g_k \cap h$ and $e^{-d_h}$ is obtained by some hyperbolic geometry considerations.

Recall that $\widetilde{\mu}_{n,j}=\sum_{i=1}^{p(n,j)} \omega_i (\mathbf{1}_{g_i}+\mathbf{1}_{\widetilde{g}_i})$ are measured laminations with finite support as previously defined. Let $\xi \in H(\widetilde X)$ and for $t \in \R$, we have
\begin{flalign*}
	\frac{d^2}{dt^2}{\mathcal{L}}([E^{t\widetilde{\mu}_{n,j}}])(\xi ) \Big{|}_{t=0}
	&= \int_{G(\widetilde{X})}\xi (h)\\
	&\quad \bigg[\sum_{i,k=1}^{p(n,j)} \omega_i \omega_k \big[\cos(g_i, h)\cos(g_k, h) - \frac{1}{2} \sin(g_i, h)\sin(g_k, h) e^{-d_h}\big]\bigg]dL_{\widetilde X}(h)\\
	&= \int_{G(\widetilde{X})} \int_{G(\widetilde{X})} \\
	&\quad \bigg\{\int_{G(\widetilde{X})}\xi (h) \big[\cos(g_i, h)\cos(g_k, h) - \frac{1}{2} \sin(g_i, h)\sin(g_k, h) e^{-d_h}\big]dL_{\widetilde X}(h)\bigg\}\\
	&\quad d\widetilde{\mu}_{n,j}(g_i) d\widetilde{\mu}_{n,j}(g_k)
\end{flalign*} where $d_h$ is the hyperbolic distance along $h$ from $g_i \cap h$ to $g_k \cap h$. Using similar arguments as in Remark \ref{rem:continuity}, $(g_i, g_k) \mapsto \int_{G(\widetilde{X})}\xi (h) \big[\cos(g_i, h)\cos(g_k, h) - \frac{1}{2} \sin(g_i, h)\sin(g_k, h) e^{-d_h}\big]dL_{\widetilde X}(h)$ is a continuous function in $(g_i, g_k)$. Since $\widetilde{\mu}_{n,j}$ converges in the weak* topology to $\widetilde{\mu}|_{K_j}$ as $n \to \infty$, similar to the first derivative it follows that
\begin{align*}
	&\lim_{n \to \infty} \int_{G(\widetilde{X})} \int_{G(\widetilde{X})} \\
	&\quad \quad \bigg\{\int_{G(\widetilde{X})}\xi (h) \big[\cos(g_i, h)\cos(g_k, h) - \frac{1}{2} \sin(g_i, h)\sin(g_k, h) e^{-d_h}\big]dL_{\widetilde X}(h)\bigg\} d\widetilde{\mu}_{n,j}(g_i) d\widetilde{\mu}_{n,j}(g_k) \\
	&= \int_{G(\widetilde{X})} \int_{G(\widetilde{X})} \\
	&\quad \quad \bigg\{\int_{G(\widetilde{X})}\xi (h) \big[\cos(g_i, h)\cos(g_k, h) - \frac{1}{2} \sin(g_i, h)\sin(g_k, h) e^{-d_h}\big]dL_{\widetilde X}(h)\bigg\} d\widetilde{\mu}|_{K_j}(g_i) d\widetilde{\mu}|_{K_j}(g_k).
\end{align*} By Theorem \ref{thm:derivative-limit}, we have that 
\begin{align*}
	\frac{d^2}{dt^2}{\mathcal{L}}([E^{t\widetilde{\mu}}])(\xi ) \Big{|}_{t=0}
	&= \lim_{j \to \infty} \int_{G(\widetilde{X})} \int_{G(\widetilde{X})} \\
	&\quad \bigg\{\int_{G(\widetilde{X})}\xi (h) \big[\cos(g_i, h)\cos(g_k, h) - \frac{1}{2} \sin(g_i, h)\sin(g_k, h) e^{-d_h}\big]dL_{\widetilde X}(h)\bigg\}\\
	&\quad d\widetilde{\mu}|_{K_j}(g_i) d\widetilde{\mu}|_{K_j}(g_k).
\end{align*}

 To simplify the notation, for the rest of the paper, we denote $A \lesssim B$ if $A \leq CB$ for some universal constant $C$. 
And now we prove
\begin{thm}
	\label{thm:earthquake-second-derivative}
	Let $\mu$ be a bounded measured lamination on a conformally hyperbolic Riemann surface $X$ and let $\widetilde{\mu}$ be its lift to the universal covering $\widetilde{X}$. Then
	\begin{flalign*}
		\frac{d^2}{dt^2}{\mathcal{L}}([E^{t\widetilde{\mu}}])(\xi ) \Big{|}_{t=0}
		&=\int_{G(\widetilde{X})} \int_{G(\widetilde{X})} \\
		&\quad \bigg\{\int_{G(\widetilde{X})}\xi (h) \big[\cos(g, h)\cos(g', h) - \frac{1}{2} \sin(g, h)\sin(g', h) e^{-d_h}\big]dL_{\widetilde 	X}(h)\bigg\}\\
		&\quad d\widetilde{\mu}(g) d\widetilde{\mu}(g')
	\end{flalign*} where $d_h$ is the hyperbolic distance along $h$ from $g \cap h$ to $g' \cap h$.
\end{thm}

	To simplify the notation, set 
	$$I(g, g')=\bigg| \int_{G(\widetilde{X})}\xi (h) \big[\cos(g, h)\cos(g', h) - \frac{1}{2} \sin(g, h)\sin(g', h) e^{-d_h}\big]dL_{\widetilde X}(h) \bigg|.$$
	In particular, $I(g, g')=0$ when $h$ intersects only $g$ or $g'$ but not both since then we have $\cos(g, h)=\sin(g, h)=0$ or $\cos(g', h)=\sin(g', h)=0$. 
	When $h$ intersects both $g$ and $g'$, we obtain an upper bound for $I(g, g')$ as follows:
	\begin{align*}
		I(g, g')
		&=\bigg| \int_{G(\widetilde{X})}\xi (h) \big[\cos(g, h)\cos(g', h) - \frac{1}{2} \sin(g, h)\sin(g', h) e^{-d_h}\big]dL_{\widetilde X}(h) \bigg|\\
		&\leq \bigg| \int_{G(\widetilde{X})}\xi (h) \cos(g, h)\cos(g', h) dL_{\widetilde X}(h)\bigg| + \bigg| \int_{G(\widetilde{X})}\xi (h) \frac{1}{2} \sin(g, h)\sin(g', h) e^{-d_h} dL_{\widetilde X}(h) \bigg|\\
		&\leq min\{Ce^{-(1+\lambda)d_g}, C'e^{-(1+\lambda)d_{g'}}\} \qquad \text{(by Lemma \ref{lem:bound-deriv} and Remark \ref{rem:bound-deriv-sin})}
	\end{align*} 
	
	It is enough to show that 
	$$\int_{G(\widetilde{X})} \int_{G(\widetilde{X})} I(g, g') d\widetilde{\mu}(g) d\widetilde{\mu}(g') < \infty.$$
%	As in the proof of Theorem \ref{thm:earthquake-derivative}, we have $$\int_{\mathcal{F}_n} I(g, g') d\widetilde{\mu}(g) < Ce^{-\lambda n}.$$ 
	We need the following estimate.
	\begin{lem}\label{lem:single-integral-estimate}
		Let $\mathcal{F}_n$ be defined as in the proof of Theorem \ref{thm:earthquake-derivative} and $I(g, g')$ be defined as above. Then $$\int_{G(\widetilde{X})} \int_{\mathcal{F}_n} I(g, g') d\widetilde{\mu}(g)d\widetilde{\mu}(g') \lesssim n e^{-\lambda n} + e^{-\lambda n},$$ where the universal constant in $\lesssim$ is the maximum of the universal constants for $ne^{-\lambda n}$ and $e^{-\lambda n}$ respectively.
	\end{lem}
	\begin{proof}
		We identify $\widetilde{X}$ with $\mathbb{D}$ by an isometry. As in the proof of Theorem \ref{thm:earthquake-derivative}, let $C_n=\{ z\in\mathbb{D}:n<\rho_{\mathbb{D}}(0,z)\leq n+1\}$ be the half-closed annulus around $0$. For each $n\geq 2$, let  $\mathcal{F}_n$ be the family of geodesics of the support $|\widetilde\mu |$ of $\widetilde\mu$ that intersect $C_n$ but do not intersect $C_{n-1}$. We partition $\mathcal{F}_n$  into finitely many subfamilies.
		
		We begin our estimate with a single subfamily $\mathcal{F}_n^i$ of $\mathcal{F}_n$ for the inner integral. Suppose that $g \in \mathcal{F}_n^i$. Since $h$ intersects both $g$ and $g'$, then $g'$ belongs to at most $2$ subfamilies of $\mathcal{F}_m$, see Figure \ref{figure:Subfamilies_for_two_geodesics}. For indexing purpose, we denote them by $\mathcal{F}_{m}^j$ and $\mathcal{F}_{m}^k$. Thus,
		\begin{align*}
			&\quad \int_{G(\widetilde{X})} \int_{\mathcal{F}_n^i} I(g, g') d\widetilde{\mu}(g)d\widetilde{\mu}(g')\\
			&=\sum_{m=1}^{\infty} \int_{\mathcal{F}_{m}^j + \mathcal{F}_{m}^k} \int_{\mathcal{F}_n^i} I(g, g') d\widetilde{\mu}(g)d\widetilde{\mu}(g')\\
			&=\sum_{m=1}^{n} \int_{\mathcal{F}_{m}^j + \mathcal{F}_{m}^k} \int_{\mathcal{F}_n^i} I(g, g') d\widetilde{\mu}(g)d\widetilde{\mu}(g') + \sum_{m=n+1}^{\infty} \int_{\mathcal{F}_{m}^j + \mathcal{F}_{m}^k} \int_{\mathcal{F}_n^i} I(g, g') d\widetilde{\mu}(g)d\widetilde{\mu}(g')\\
			&= A + B
		\end{align*} 
		\begin{figure}[h!]
			\centerline{
				\includegraphics[width=0.3\textwidth]{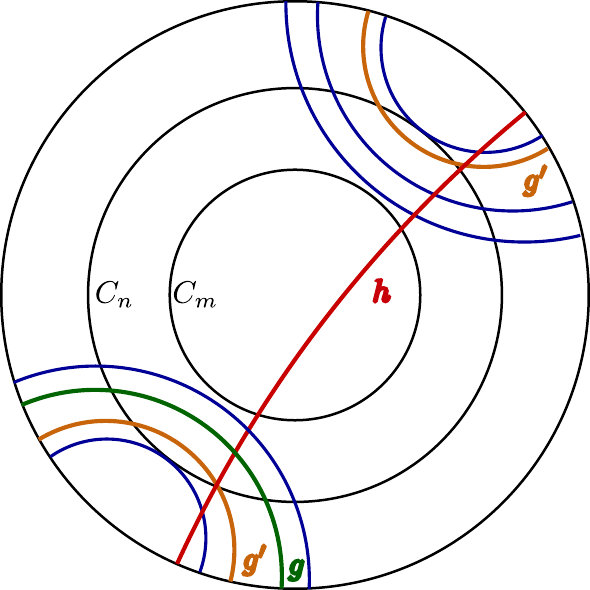}
			}
			\caption{\small{$h$ intersecting geodesics $g$ and $g'$.}}\label{figure:Subfamilies_for_two_geodesics}
		\end{figure}
		
		Note that the total measure of the geodesics in $\mathcal{F}_{m}^j + \mathcal{F}_{m}^k$ is at most $2\|\widetilde\mu\|_{Th}$. Thus,
		$$A\leq 2\Vert \widetilde \mu \Vert_{Th} C n 
		e^{-(1+\lambda)n}
		\leq C_1 n 
		e^{-(1+\lambda)n}.$$ To estimate $B$, we consider 
		\begin{align*}
			\sum_{m=n+1}^{\infty} \int_{\mathcal{F}_{m}^j + \mathcal{F}_{m}^k} \int_{\mathcal{F}_n^i} I(g, g') d\widetilde{\mu}(g)d\widetilde{\mu}(g')
			&=\int_{\mathcal{F}_n^i} \bigg(\sum_{m=n+1}^{\infty} \int_{\mathcal{F}_{m}^j + \mathcal{F}_{m}^k} I(g, g') d\widetilde{\mu}(g') \bigg) d\widetilde{\mu}(g)\\
			&\leq \int_{\mathcal{F}_n^i} \bigg(2\sum_{m=n+1}^{\infty} Ce^{-(1+\lambda)m} \bigg) d\widetilde{\mu}(g)\\
			&\leq 2\Vert \widetilde \mu \Vert_{Th} C \frac{e^{(1+\lambda)}}{e^{(1+\lambda)}-1} e^{-(1+\lambda)(n+1)}\\
			&\leq C_2 
			e^{-(1+\lambda)(n+1)}.
		\end{align*} Then,
			 $$\int_{G(\widetilde{X})} \int_{\mathcal{F}_n^i} I(g, g') d\widetilde{\mu}(g)d\widetilde{\mu}(g') 
			 \lesssim n 
			 e^{-(1+\lambda)n} + e^{-(1+\lambda)(n+1)},$$ where the universal constant in $\lesssim$ is the $max\{C_1, C_2\}$.
			 Since the number of subfamilies $\mathcal F_n^i$ is at most $\frac{2\pi}{C} e^n$, 
			 \begin{align*}
			 	\int_{G(\widetilde{X})} \int_{\mathcal{F}_n} I(g, g') d\widetilde{\mu}(g)d\widetilde{\mu}(g')
			 	&\lesssim \frac{2\pi}{C}e^n \bigg(n e^{-(1+\lambda)n} + e^{-(1+\lambda)(n+1)}\bigg)\\
			 	&\lesssim n e^{-\lambda n} + e^{-\lambda n}.
			 \end{align*}
	\end{proof}
	\begin{proof}[Proof of Theorem \ref{thm:earthquake-second-derivative}]
		It follows from Lemma \ref{lem:single-integral-estimate}, 
		\begin{align*}
			\int_{G(\widetilde{X})} \int_{G(\widetilde{X})} I(g, g') d\widetilde{\mu}(g) d\widetilde{\mu}(g')
			&=\sum_{n=1}^{\infty} \int_{G(\widetilde{X})} \int_{\mathcal{F}_n} I(g, g') d\widetilde{\mu}(g)d\widetilde{\mu}(g')\\
			&\lesssim \sum_{n=1}^{\infty} (n e^{-\lambda n} + e^{-\lambda n}) < \infty
		\end{align*} and the theorem is proved.
	\end{proof}
	
	By changing the basepoint of $\mathcal{T}(X)$, we obtain an immediate consequence of Theorem \ref{thm:earthquake-second-derivative}.
	\begin{cor}
			\label{thm:earthquake-second-derivative_for_any_t}
		Let $\mu$ be a bounded measured lamination on a conformally hyperbolic Riemann surface $X$ and let $\widetilde{\mu}$ be its lift to the universal covering $\widetilde{X}$. Then
		\begin{flalign*}
			\frac{d^2}{dt^2}{\mathcal{L}}([E^{t\widetilde{\mu}}])(\xi )
			&=  \int_{G(\widetilde{X})} \int_{G(\widetilde{X})} \\
			&\quad \bigg\{\int_{G(\widetilde{X})}\xi (h) \big[\cos(E^{t\widetilde{\mu}}(g), E^{t\widetilde{\mu}}(h))\cos(E^{t\widetilde{\mu}}(g'), E^{t\widetilde{\mu}}(h))\\
			&\quad - \frac{1}{2} \sin(E^{t\widetilde{\mu}}(g), E^{t\widetilde{\mu}}(h))\sin(E^{t\widetilde{\mu}}(g'), E^{t\widetilde{\mu}}(h)) e^{-d_{E^{t\widetilde{\mu}}(h)}}\big]dL_{[E^{t\widetilde{\mu}}]}(h)\bigg\} \\
			&\quad d\widetilde{\mu}(g) d\widetilde{\mu}(g')
		\end{flalign*} where $d_{E^{t\widetilde{\mu}}(h)}$ is the hyperbolic distance from $E^{t\widetilde{\mu}}(g) \cap E^{t\widetilde{\mu}}(h)$ to $E^{t\widetilde{\mu}}(g') \cap E^{t\widetilde{\mu}}(h)$ along the geodesic $E^{t\widetilde{\mu}}(h)$.
	\end{cor}
\begin{rem}\label{rem:quake-fail}
	We would like to find the second derivative along a quake-bend of the Liouville measure in a neighborhood of the Teichm\"uller space $\mathcal{T}(X)$ similar to Theorem \ref{thm:quake-der} (given by a limit). However, the quantity $d_{E^{\tau\widetilde{\mu}|_{K_j}}(h)}$ is not well-defined in the upper half-space $\mathbb{H}^3$ since there are no such intersection points $E^{\tau\widetilde{\mu}|_{K_j}}(g) \cap E^{\tau\widetilde{\mu}|_{K_j}}(h)$ and $E^{\tau\widetilde{\mu}|_{K_j}}(g') \cap E^{\tau\widetilde{\mu}|_{K_j}}(h)$.
\end{rem}

\end{document}